\documentclass[11pt]{amsart}
\usepackage{amsmath, amsthm, amssymb}
\usepackage{amsfonts}
\usepackage{fullpage}
\usepackage{color}
\usepackage{hyperref}
\usepackage{enumitem}
\usepackage{bm}
\usepackage{verbatim}

\usepackage[dvipsnames]{xcolor}
\usepackage{graphicx}
\DeclareGraphicsExtensions{.pstex,.eps}

\usepackage{xcolor}

\renewcommand{\ell}{{l}}  

\newcommand{\bE}{\mathbf{E}}
\newcommand{\bR}{\mathbf{R}}

\newcommand{\R}{{\mathbb{R}}}
\newcommand{\Z}{{\mathbb Z}}

\renewcommand{\S}{{\mathbb S}}

\renewcommand{\AA}{{\mathcal A}}

\newcommand{\EE}{{\mathcal E}}

\newcommand{\MM}{{\mathcal M}}

\newcommand{\DD}{{\mathcal D}}

\newcommand{\QQ}{{\mathcal Q}}
\newcommand{\RR}{{\mathcal R}}

\newcommand{\XX}{{\mathcal X}}

\newcommand{\tF}{{\tilde{F}}}

\newcommand{\td}{{\tilde d}}

\newcommand{\tp}{{\tilde p}}
\newcommand{\tOm}{{\tilde \Om}}
\newcommand{\tS}{{\tilde S}}

\newcommand{\tr}{\operatorname{tr}}

\newcommand{\lf}{L^\infty}

\renewcommand{\div}{\operatorname{div}}

\theoremstyle{plain}

\newtheorem{thm}{Theorem}[section]
\newtheorem{prop}[thm]{Proposition}
\newtheorem{cor}[thm]{Corollary}
\newtheorem{lemma}[thm]{Lemma}

\theoremstyle{definition}

\newtheorem{rem}{Remark}[thm]

\numberwithin{equation}{section}

\def\squarebox#1{\hbox to #1{\hfill\vbox to #1{\vfill}}}

\newcommand{\<}{\langle}
\renewcommand{\>}{\rangle}

\renewcommand{\d}{\partial}
\newcommand{\ep}{\epsilon}
\newcommand{\lV}{\lVert}
\newcommand{\rV}{\rVert}

\def\de{\delta}
\def\De{\Delta}
\def\ep{\epsilon}

\def\la{\lambda}
\def\La{\Lambda}
\def\si{\sigma}

\def\om{\omega}
\def\Om{\Omega}

\def\nab{\nabla}

\def\al{\alpha}
\def\les{\lesssim}

\title[Liquid crystal elastomers]
{On the hydrodynamics of hyperbolic liquid crystal elastomers}
	
\author[X. Hao]{Xiaonan Hao$^1$}
\author[J. Huang]{Jiaxi Huang$^2$}
\author[N. Jiang]{Ning Jiang$^{3}$}

\address{$^1$School of Mathematical Sciences, Peking University,
	\newline\indent
	Beijing
	100871, P.R. China}
\email{\href{mailto:xn\_hao@163.com}{xn\_hao@163.com}}	
\address{$^2$School of Mathematics and Statistics, Beijing Institute of Technology,
		\newline\indent
				Beijing
100081, P.R. China}
\email{\href{mailto:jiaxih@bit.edu.cn}{jiaxih@bit.edu.cn}}
\address{$^3$School of Mathemtical and Statistics, Wuhan University,
	\newline\indent
	Wuhan 430072, P. R. China}
\email{\href{mailto:njiang@whu.edu.cn}{njiang@whu.edu.cn}}

\subjclass[2010]{}

	
\begin{document}
		
\begin{abstract}
     Liquid crystal elastomers are special cross-linked polymer materials combining
     the large elastic deformability of elastomers with the orientational orders of liquid crystals. 
     This model exhibits markedly different phenomena than the liquid crystal model due to the strong coupling between mechanical elastic deformation and orientation vector. 
    
     Our results are threefold. (i) First we derive the hydrodynamics of liquid crystal elastomers with inertial effect by energetic variational approach inspired by Chun Liu \cite{Liu-2009}.
     (ii) Then we study the hyperbolic liquid crystal elastomers from mathematical point, which is a fully quasilinear hyperbolic system.  
     The local well-posedness for large data is proved by zero-viscosity limit method. (iii) Finally, we show the global regularity for small and smooth initial data near the constant equilibrium in three dimensions, which is achieved by carefully investigating the structure of second-order material derivatives and the cancellations in the system. 
     This is the first global result on hyperbolic liquid crystal elastomers.
\end{abstract}

\date{\today}
\maketitle

		
\setcounter{tocdepth}{1}

\section{Introduction}
Liquid crystal elastomers (LCEs) are fascinating materials, which were first proposed by de Gennes \cite{deGennes,deGennes-1975} in 1969 and first synthesized by K\"upfer-Finkelmann \cite{KF-1991} in the early 1990s, see also the related works \cite{FKR-1981,WT-2007}. This kind of soft smart material exhibits many new phenomena not found in either liquid crystals or polymers, and has many potential applications.

Physically, liquid crystal elastomers are materials that combine the liquid state and anisotropic ordering properties of liquid crystals with elasticity features of rubber-like solids. This combination exhibits new behaviors well beyond the modified liquid crystal or the elastic solid.
A typical characteristic of LCEs is the strong coupling structure between the mechanical deformation and the orientational order.
Due to the interaction, the change of the liquid crystal order can induce macroscopic change in shape, while the orientational order can be also affected by mechanical deformation.
In 1969, de Gennes \cite{deGennes} first proposed the LCEs to solve the pioneering question: \emph{whether cross-linking conventional polymers (different from liquid crystal ones) into a network in the presence of a liquid crystal solvent results into an (anisotropic) liquid crystalline polymer}.
The answer turned out to be negative, that is, conventional rubber does not possess any memory of the anisotropic environment present at cross-linking, see \cite{WT-2007}. 
This shows that the LCEs is a kind of materials different from liquid crystals or polymers, and admits some exceptional physical properties.
In fact, thanks to their excellent actuation performances, the LCEs can be used in plenty of technological applications, including soft actuators, artificial muscles, flexible robots, deformable lenses, stimuli-responsive surfaces, and so on. One please refers to \cite{J-2012,WT-2007} for more details.

Because of the wide range of applications, a better understanding of the intrinsic dynamic behavior of these materials is needed before they can be exploited on an industrial scale.
In \cite{deGennes-1975} de Gennes established the frame of energy method by introducing the Landau-de Gennes free energy density $f_{LdG+Frank}$ and the elastic free energy density $f_{elast}$.
Terentjev and Warner in \cite{WT-2007} made important advances on the continuum theory of liquid crystal elastomers. Jointly with Bladon, they \cite{BTW-1993} proposed the well-known Bladon-Ternetjev-Warner liquid crystal elastomer energy, which is a Gaussian energy of neo-Hookean type, that is quadratic in the gradient of deformation and incorporates the liquid crystal anisotropy.

However, compared with the extensive physical experiments, the mathematical research on liquid crystal elastomers is still in its infancy. So far, there are some works \cite{DD-2000,CDD-2002,DD-2002,CLY-2006,Luo-2010,LC-2012,CGL-2014} concerning on the existence of minimizer of the energy for some proposed models, such as Bladon-Terentjev-Warner model, which is studied by variational methods. Then these results are applied to the analysis and numerical simulation of the phenomena found in the experiment.
In \cite{ZSP-2011,Z-2014}, they considered a continuum model describing the dynamic behavior of nematic liquid crystal elastomers and supplied a numerical scheme to solve the governing equations.

In this article, due to the important applications of LCEs, we will focus on the hydrodynamics of liquid crystal elastomers. 
First, we derive the LCEs model with inertial effect by energetic variational approaches. Then we pay our attentions to the long-time behaviors of LCEs, especially the model without viscosity and without damping. In this setting, the direction of Liquid crystal molecules evolves under a highly nontrivial metric, which has appeared in many areas of fluids and attracts our great attention. With the investigation of structure, we establish the global well-posedness of LCEs for small data in three dimensions.

\subsection{Hydrodynamics of liquid crystal elastomers}
The hydrodynamics of LCEs are the key to understanding the long-time behaviors of LCEs. 
Here inspired by the excellent works of Chun Liu and his collaborators in \cite{Liu-2009,LLZ-CPAM2005,Wu-Xu-Liu-2013,WL-2022,GKL-2018}, our first objective is to derive the hydrodynamics of liquid crystal elastomers with inertial effect by energetic variational approaches.
The approaches are motivated by the seminal work of
Rayleigh \cite{Ra} and Onsager \cite{On1,On2}. The framework, including Least Action
Principle and Maximum Dissipation Principle, provides a unique, well-defined,
way to derive the coupled dynamical systems from the total energy functionals and
dissipation functions in the dissipation law.  
We will exhibit the derivation of LCEs in Section \ref{sec-variation} in detail.

Now, we introduce the LCEs system that we will consider in this article. The LCEs consists of the following equations of the
velocity field $u(x, t)\in \R^n$, the deformation gradient $F(x,t)\in \R^n\times \R^n$ and the direction field $d(x, t) \in \S^{n-1}$, and $(x, t) \in \R^n\times \R^+$:
\begin{equation}\label{PHLC}
	\begin{aligned}
		\left\{ \begin{aligned}
			&\partial_t u + u \cdot \nabla u  + \nabla p = \nab\cdot\big(FF^{T}\big)- \div (\nabla d \odot \nabla d) + \div \sigma\,, \\
			&\d_t F+u\cdot\nab F=\nab u F\,,\\
			&\rho_1 \ddot{d} = \Delta d + \Gamma d + \lambda_1 (\dot{d} + B d) + \lambda_2 A d\,,
		\end{aligned}\right.
	\end{aligned}
\end{equation}
with the constraints
\begin{align} \label{constraints1}
	&\div u=0,\quad \nab\cdot F^T=0,\quad
	\sum_m F_{mj}\d_m F_{ik}=\sum_m F_{mk}\d_m F_{ij},\quad
	|d|=1.
\end{align}
We should mention that the first two constraints in \eqref{constraints1} are the consequences of the incompressibility condition. The third compatibility condition is from the Lagrangian derivatives commute, i.e. $\d_{X_k}\tilde F_{ij}=\d_{X_j}\tilde F_{ik}$ with $\tilde F(X,t)=F(x(X,t),t)$. The last one in \eqref{constraints1} comes from $d(t,x)\in\S^2$ which is a unit vector. A detailed explanation of the constraints can be found in \cite{LLZ}.

In the above system, we denote the component $F_{ij}$ as the element in the $i$-th row and $j$-th column of matrix $F$, and denote $F^T$ as the transpose of matrix $F$. We will adopt the notations of
\begin{align*}
	(\nab \cdot F)_i=\d_j F_{ij},\quad (\nab u)_{ij}=\d_j u_i.
\end{align*}
The superposed dot denotes the material
derivative $\d_t+u\cdot\nab$.
The second nonlinear term in $u$-equation is given by
\[(\nabla d \odot \nabla d)_{ij} = \sum_{k=1}^n\partial_i d_k \partial_j d_k, \quad \big(\div (\nabla d \odot \nabla d)\big)_i = \sum_{j,k=1}^n \d_j(\partial_i d_k \partial_j d_k).\]

We denote the rate of strain tensor $A$ and skew-symmetric part $B$ of the strain rate as
\begin{equation*}
	\begin{aligned}
		A = \tfrac{1}{2}(\nabla u + \nabla^\top u)\,,\quad B= \tfrac{1}{2}(\nabla u - \nabla^\top u)\,,\end{aligned}
\end{equation*}
whose components are given by $A_{ij} = \tfrac{1}{2} (\partial_j u_i + \partial_i u_j)$, $B_{ij} = \tfrac{1}{2} (\partial_j u_i - \partial_i u_j)$.
We also define $N = \dot d + B d$ as the rigid rotation part of director changing rate by fluid vorticity, where $(B d)_i =B_{ki} d_k$.

The stress tensor $\sigma$ has the following form:
\begin{equation}\label{Extra-Sress-sigma}
	\begin{aligned}
		\sigma_{ji}=  \nu_1 d_k A_{kp}d_p  d_i d_j + \nu_2  d_j N_i  + \nu_3 d_i N_j  + \nu_4 A_{ij} + \nu_5 A_{ik}d_k d_j   + \nu_6 d_i A_{jk}d_k \,.
	\end{aligned}
\end{equation}
These coefficients $\nu_i (1 \leq i \leq 6)$ which may depend on material and temperature, are usually called Leslie coefficients, and are related to certain local correlations in the fluid. Usually, the following relations are frequently introduced in the literatures \cite{Ericksen-1961-TSR,Leslie-1968-ARMA, Wu-Xu-Liu-2013}.
\begin{equation}\label{Coefficients-Relations}
	\lambda_1=\nu_2-\nu_3\,, \quad\lambda_2 = \nu_5-\nu_6\,,\quad \nu_2+\nu_3 = \nu_6-\nu_5\,.
\end{equation}
The first two relations are necessary conditions in order to satisfy the equation of motion identically, while the third relation is called {\em Parodi's relation}, which is derived from Onsager reciprocal relations expressing the equality of certain relations between flows and forces in thermodynamic systems out of equilibrium. Under Parodi's relation, we see that the dynamics of an incompressible nematic liquid crystal flow involve five independent Leslie coefficients in \eqref{Extra-Sress-sigma}.

In \eqref{PHLC}, the coefficient $\rho_1 > 0$ is the inertial constant, and the Lagrangian multiplier $\Gamma$ is (which ensures the geometric constraint $|d|=1$):
\begin{equation}\label{Lagrange-Multiplier}
	\Gamma = - \rho_1 |\dot{d}|^2 + |\nabla d|^2 - \lambda_2 d^\top A d\, .
\end{equation}
We remark that the last two terms in the third equation of \eqref{PHLC} is the so-called kinematic transport, i.e.
\begin{equation}\label{kinematic transport}
	g= \lambda_1 (\dot{d} + B d) + \lambda_2 A d\,,
\end{equation}
which represents the effect of the macroscopic flow field on the microscopic structure. The material coefficients $\lambda_1$ and $\lambda_2$ reflect the molecular shape and the slippery part between the fluid and particles. The first term represents the rigid rotation of the molecule, while the second term stands for the streching of the molecule by the flow.

The LCEs system \eqref{PHLC} has close relation to the traditional liquid crystals system and elastodynamics. Particularly, in this paper we consider the liquid crystal elastomers with inertial effect. In other words, we build up our LCEs models on the so-called (parabolic)\footnote{when the viscosity term $\mu\Delta u$ is involved, we name it parabolic-hyperbolic, while for the invisid case $\mu=0$, as we consider in this paper, it is called hyperbolic.}-hyperbolic hydrodynamic models of Ericksen and Leslie. More specifically, when the elasiticity $F=I$, the LCEs system reduces to Ericksen-Leslie's hyperbolic liquid crystal model, which was introduced by Ericksen and Leslie in their pioneering works \cite{Ericksen-1961-TSR} and \cite{Leslie-1968-ARMA} respectively in the 1960's.
The parabolic-hyperbolic liquid crystal is a nonlinear coupling of incompressible Navier-Stokes equations (Euler equations for invisicd case) to wave map system with target manifold $\mathbb{S}^2$.
Jiang-Luo established in \cite{Jiang-Luo-2018} the well-posedness in the context of classical solutions. Furthermore, with an additional assumption on the coefficients which provides a damping effect, and the smallness of the initial energy, the unique global classical solution was established.
For the simplified Ericksen-Leslies system, Cai-Wang \cite{CW}  proved the global regularity of the  hyperbolic liquid crystal near the constant equilibrium by employing the vector field method.
Later, the general incompressible Ericksen-Leslies system without kinematic transport and simplified compressible hyperbolic liquid crystal model was considered by Huang-Jiang-Luo-Zhao in \cite{HJLZ1,HJLZ2}. They proved the global regularity and scattering near the equilibrium by space-time resonance method. Recently, Huang-Jiang-Zhao  \cite{HJZ-2023} proved the almost global well-posedness for  the simplified Liquid crystal in two dimensions.
Finally, the important special case of liquid crystal,
is the so-called parabolic Ericksen-Leslie system, which was extensively studied from the mid 1980s. For a more complete review of the works, please see the reference \cite{LW-survey}.

When the direction $d$ of liquid crystal molecule is a constant unit vector, the LCEs system reduces to incompressible elastodynamics, which have been extensively studied. Lin-Liu-Zhang \cite{LLZ-CPAM2005} and Chen-Zhang \cite{CZ-2006} studied the the global existence of incompressible elastodynamics for small data. Sideris-Thomases \cite{ST-2005,ST-2007} and Lei \cite{Lei2016} investigated the corresponding inviscid case. In a different way, Wang \cite{Wang-2017} established the global existence and the asymptotic behavior for the 2D incompressible isotropic elastodynamics for sufficiently small and smooth initial data in the Eulerian coordinates formulation.
Subsequently, Cai-Lei-Lin-Masmoudi \cite{CLLM-CPAM2019} showed  the vanishing viscosity limit for incompressible viscoelasticity.
For the more complete researches in this direction, one please refers to \cite{Lin-CPAM2012}.

To analyze the behaviors of LCEs from mathematical point, we focus on the special but important model of liquid crystal elastomers: set $\nu_i=0$ for $i=1,\cdots,6$, and hence $\la_1=\la_2=0$, $\rho_1=1$. Then the system \eqref{PHLC} is reduced to a fully hyperbolic nonlinear system. More precisely, we study the liquid crystal elastomer in the form
\begin{equation}        \label{ori_sys}
	\left\{
	\begin{aligned}
		&\d_t u+u\cdot \nab u+\nab p=\nab\cdot(FF^{T})- \div  (\nab d\odot\nab d),\\
		&\d_t F+u\cdot\nab F=\nab u F,\\
		&D_t^2 d-\De d=(-|D_t d|^2+|\nab d|^2)d,\\
		&(u,F,d,\d_t d)\big|_{t=0}=(u_0,F_0,d_0,d_1),
	\end{aligned}
	\right.
\end{equation}
on $\R^n\times \R^+$ with the constraints
\begin{align} \label{constraints1-re}
	&\div u=0,\quad \nab\cdot F^T=0,\quad
	\sum_{m}F_{mj}\d_m F_{ik}=\sum_m F_{mk}\d_m F_{ij},\quad
	|d|=1.
\end{align}

In summary, the LCEs system \eqref{ori_sys}-\eqref{constraints1-re} studied in the current paper has three new features, compared to the previous related models.
Firstly, the system contains both liquid crystal (Ericksen-Leslie) hydrodynamics and elastodynamics.
Secondly, it is coupled by the incompressible Euler, rather than the Navier-Stokes system, which is inviscid.
Thirdly, it has inertial effect, which involves second order material derivatives.
All of these features characterize \eqref{ori_sys} as a fully hyperbolic system with high nonlinearities and nontrivial geometric structure.

\subsection{The main results}
Once the LCEs is derived, our second objective in this paper is to understand 
the behaviors of LCEs. Here we start with the local well-posedness of \eqref{ori_sys}-\eqref{constraints1-re} for large data, which is proved by zero-viscosity limit method.
More importantly, we establish the global in time well-posedness for small initial data in three dimensions, where the cancellations and the structures in the system play crucial roles.

The following is our first main theorem, whose proof is provided in the Appendix \ref{local}.

\begin{thm}[Local well-posenedss for large data in dimensions $n\geq 2$]\label{LWP}
	Let dimensions $n\geq 2$, integer $k_0=3$, and let the initial data satisfy $(u_0,F_0,\nab d_0,d_1)\in H^{k_0}(\R^n)$, $|d_0|=1$, $d_0\cdot d_1=0$. The initial energy is defined as $E_\si^{in}=\frac{1}{2}(\|u_0\|_{H^{\si}}^2+\|F_0\|_{H^{\si}}^2+\|\nab d_0\|_{H^{\si}}^2+\|d_1\|_{H^{\si}}^2)$ for any $\si\geq k_0$. If the initial energy $E_{k_0}^{in}<\infty$, then there exists $T>0$, depending only on $E_{k_0}^{in}$, such that the Cauchy problem of the system \eqref{ori_sys}-\eqref{constraints1-re} admits a unique solution $(u,F,d)$ such that $(u,F,\nab d,D_t d)\in L^\infty(0,T;H^s)$. Moreover, the solution $(u,F,d)$ satisfies
	
	i) Energy estimates:
	\begin{align}   \label{LE-local}
		\sup_{0\leq t\leq T}\Big(\|u\|_{H^{k_0}}^2+\|F\|_{H^{k_0}}^2+\|\nab d\|_{H^{k_0}}^2+\|D_t d\|_{H^{k_0}}^2\Big)(t)\leq C_0,
	\end{align}
	where the positive constant $C_0$ depends only on $E_{k_0}^{in}$.
	
	ii) Propagation of regualrity: for any $\si\geq k_0$, we have
	\begin{align*}
		\sup_{0\leq t\leq T}\Big(\|u\|_{H^\si}^2+\|F\|_{H^\si}^2+\|\nab d\|_{H^\si}^2+\|D_t d\|_{H^\si}^2\Big)(t)\leq C_\si E^{in}_\si.
	\end{align*}
\end{thm}

Our second main theorem is as follows. Here the precise definitions of the functional spaces will be given in Section \ref{sec-2.2}.
\begin{thm}[Global well-posedness for small data in three dimensions]           \label{Ori_thm0}
	Let integer $N\geq 9$, $0<\de\leq 1/8$ and $M>0$ be three given constants and the initial data $(u_0,F_0-I,d_0,d_1)\in H^N_\Lambda$. Then there exists $\ep>0$ sufficiently small depending on $N$, $\de$ and $M$ such that, for all initial data $(u_0,F_0,d_0,d_1)$ satisfying the constriants \eqref{constraints1} with
	\begin{equation}  \label{Main-ini0}
		\|(u_0,F_0-I,d_0,d_1) \|_{H^N_\Lambda}\leq M,\quad \|(u_0,F_0-I,d_0,d_1) \|_{H^{N-2}_\Lambda}\leq \ep,
	\end{equation}
the hyperbolic liquid crystal elastomers system \eqref{ori_sys}-\eqref{constraints1-re} is global in time well-posed in three dimensions. Moreover, for all $t\in[0,+\infty)$ the solution $(u,F-I,d)\in H^{N}_Z$ has the bounds
	\begin{equation}  \label{energybd0}
		E^{1/2}_N(u,F-I,d;t)\les \<t\>^\de,\quad E^{1/2}_{N-2}(u,F-I,d;t)\les \ep.
	\end{equation}
\end{thm}

\begin{rem}	
Different from parabolic-hyperbolic liquid crystal \cite{Jiang-Luo-2018}, the model \eqref{ori_sys} is a fully nonlinear hyperbolic system, and hence the null structure of nonlinearities (especially the second-order material derivative) plays a more important role. Precisely, the viscosity term $\mu\De u$ in $u$-equation brings many benefits in decay estimates and energy estimates, which also enable ones to control the direction $d$ well. However, due to the lack of viscosity in the current paper, we must turn our attention to the null structure of the second-order material derivative in $d$-equation. 
\end{rem}

Since the system are considered near constant equilibrium $(\vec{0},I,\vec{i})$, it is natural to reformulate the system \eqref{ori_sys}.
Precisely, we denote
\begin{equation*}
	H=F-I,\quad d=(\cos \phi_1\cos\phi_2,\sin\phi_1\cos\phi_2,\sin\phi_2),
\end{equation*}
where the $\phi_1$ represents the angle between
the $x$-axis and the projection of $d$ onto the $x-y$ plane, and $\phi_2$ represents the angle
between $d$ and the $x-y$ plane.
Then the system \eqref{ori_sys} is rewritten as
\begin{equation}       \label{EL-elas-1}
	\left\{
	\begin{aligned}
		&\d_t u-\nab\cdot H+\nab p=-u\cdot \nab u+\nab\cdot(HH^T)- \div  (\nab \phi\odot \nab \phi)+\mathcal R_1,\\
		&\d_t H-\nab u =-u\cdot \nab H+\nab u H,\\
		&D_t^2 \phi-\De \phi=\mathcal R_2,\\
		&(u,H,\phi,\d_t\phi)(x,0)=(u_0,H_0,\psi_{0},\psi_{1})(x),
	\end{aligned}
	\right.
\end{equation}
with $\phi=(\phi_1,\phi_2)$, $\psi_i=(\psi_{i,1},\psi_{i,2})$ for $i=0,1$,  and the constraints
\begin{align}     \label{cst-re}
	&\div u=0,\quad \nab\cdot H^T=0,\\   \label{curl}
	&\d_j H_{ik}-\d_k H_{ij}=\sum_m (H_{mk}\d_m H_{ij}-H_{mj} \d_m H_{ik}),
\end{align}
where $\mathcal R_1$ and $\mathcal R_2$ are higher order terms
\begin{align}  \label{R1}
	&\mathcal R_1=\sum_j\d_j(\sin^2\phi_2 \nab\phi_1 \d_j \phi_1),\\\label{R2}
	&\mathcal R_2:=\Big(2\tan\phi_2(D_t\phi_1 D_t\phi_2-\nab\phi_1\nab\phi_2),\frac{1}{2}\sin 2\phi_2(-|D_t\phi_1|^2+|\nab\phi_1|^2)\Big)^T.
\end{align}

For the reformulated system \eqref{EL-elas-1}, we have the following result:
\begin{thm}           \label{Ori_thm}
Let integer $N\geq 9$, $0<\de<1/8$ and $M>0$ be three given constants and $(u_0,H_0,\psi_{0},\psi_{1})\in H^N_\Lambda$. Then there exists $\ep>0$ sufficiently small depending on $N,\ \de$ and $M$ such that, for all initial data $(u_0,H_0,\psi_{0},\psi_{1})$ satisfying the constriants \eqref{cst-re} and \eqref{curl} with
	\begin{equation}\label{MainAss_dini}
		\|(u_0,H_0,\psi_{0},\psi_{1}) \|_{H^N_\Lambda}\leq M,\quad \|(u_0,H_0,\psi_{0},\psi_{1}) \|_{H^{N-2}_\Lambda}\leq \ep,
	\end{equation}
	the reformulation of liquid crystal elastomers system \eqref{EL-elas-1} is global in time well-posed in three dimensions. Moreover, for all $t\in[0,+\infty)$ the solution $(u,H,\phi)\in H^{N}_Z$ has the bounds
	\begin{equation}     \label{energybd}
		E^{1/2}_N(u,H,\phi;t)\les \<t\>^\de,\quad E^{1/2}_{N-2}(u,H,\phi;t)\les \ep.
	\end{equation}
\end{thm}

\begin{rem}
	Here we introduce the new variables, angles $\phi$, in order to distinguish the nonlinear terms. In fact, the nonlinear terms $(-|D_t d|^2+|\nab d|^2)d$ are essentially cubic terms, which has been discussed by Tataru \cite{Tataru} and Tao \cite{Tao2} using gauge theory or renormalization. This feature can also be exhibited by rewriting the equations of $d$ in terms of new variables $\phi$ in our article. Hence, the main obstacle in proving global existence is the quadratic terms from second-order material derivatives.
\end{rem}

\subsection{Main ideas and novelties}
The main strategy to prove global regularity relies on an interplay between the control of high order energies and decay estimates, which is based on the vector-fields method and weighted energy method.
Our proof includes the following three main ingredients:
\begin{itemize}
	\item[i)] \emph{Decay estimates}.
	The decay estimates containing the time and space variables play crucial role in dealing with the energy estimates, especially for the convective term near the light cone.

	\item[ ii)] \emph{Weighted $L^2$-bounds}.
	The weighted $L^2$-energy is a bridge that connects decay estimates and energy estimates, whose proof relies on  the Klainerman-Sideris type estimate and the special structure of nonlinearities.

	\item[iii)] \emph{Energy estimates}.
	With the help of the above estimates,  we  show a robust energy that better matches the cancellations and structures of the equations. In fact, the desired energy can be obtained by the smallness of the low-order energy estimates.  Many  new problems arise at this stage.
\end{itemize}

There are several new difficulties need to be dealt with. Firstly, the second-order material derivative appearing in the $\phi$-equation brings obstacle for us to enclose  the energy estimates.  This requires more refined analysis for the material derivative. Secondly, the phenomena of derivative loss appears frequently in the system. We have to investigate the nonlinearities and adjust the energy norms carefully. Finally, due to the decays of wave equations are $\<t\>^{-1}$, it may generate log-type growth that prevents the global well-poseness. Thus we need dig out  the null structure of the system to gain more decays.

To get around the above difficulties, some key observations and novelties are emphasized as follows:

\emph{1. Second-order material derivative.} Essentially, the $\phi$-equations in \eqref{EL-elas-1} are cubic wave equations with high nontrivial metric, which is called \emph{acoustical metric} dependingon the velocity $u$. The acoustical metric (or second-order material derivative) will cause great obstruction towards to the global well-posedness. Fortunately, thanks to divergence free condition, the quadratic terms $u\cdot\nab\d_t\phi$ and $\d_t u\!\cdot\!\nab\phi$ from second-order material derivative can gain enough decay rates.  More precisely, from the good unknown ``$\om\d_t\phi+\nab\phi$" or the decomposition $\nab=\om\d_r-\frac{\om\times\Om}{r}$, we have the dompositions
\[ u\cdot\nab\d_t\phi=u\cdot(\om\d_t^2\phi+\nab\d_t\phi)-u\cdot\om \d_t^2\phi=u\cdot \om\d_r \d_t\phi-u\cdot \frac{\om\times\Om}{r}\d_t\phi,           \]
\[  \d_t u\cdot\om \d_r\phi=\d_t u\cdot \om\d_r \phi-\d_t u\cdot \frac{\om\times\Om}{r}\phi.         \]
These together with the decays estimates of good unknown and good quantities $u\cdot\om$, $\d_t u\cdot\om$ benefit a decay of $\<t\>^{-3/2}$ in our setting.

\medskip
\emph{2. Derivative loss and energy norms.}
To deal with the derivative loss problem, we have to analyze the strong coupling structure of system  carefully. In fact, the angles $\phi$ evolve on the background manifold $(\R^3,g)$, then from geometric point,  we should define the energy norms intrinsically that match the metric. Inspired by the observation, the energy norms of $\phi$ are given by
\[   \int_{\R^3} (|\d_t \tilde\phi|^2+g^{ij}\d_i \tilde\phi \d_j\tilde\phi) \sqrt{\det g} \ dX=\int_{\R^3} (|D_t \phi|^2+|\nab \phi|^2) \ dx,       \]
where $X=(X_1,X_2,X_3)$ is the Lagrangian coordinates and $\tilde\phi(X,t)=\phi(x(X,t),t)$.
The energy solves the main part of issues cased by the derivative loss. To deal with the remainder issues, we shall further make energy corrections that equivalent to the above energy. In conclusion, a key rule is that we should work with a robust energy
which is both coercive and propagates well along the flow.

\medskip
\emph{3. Nonlinear structure}. The LCEs system \eqref{EL-elas-1} admits good quantities $\om\cdot u$ and $\om_k H_{kj}$, which allow us to gain more decays in energy estimates. Precisely, the quadratic terms satisfy the following forms
\[  \d_k U H_{kj},\qquad u\cdot\nab U,    \]
where $U\in \{u,\ H, \nab\phi\}$ denotes the unknowns. Inspired by $\nab=\om\d_r-\frac{\om\times\Om}{r}$, we have the decomposition
\[ \d_k U H_{kj}= \d_r U \om_k H_{kj}-\frac{(\om\times \Om)_k}{r}U H_{kj}.   \]
The first term in $L^2$ is shown to have $\<t\>^{-3/2}$ due to divergence free, and the second term admits a better decay. We note that the other form $u\cdot\nab U$ can be handled similarly. Hence, we can bound the above quadratic terms well.

\subsection{An outline of the paper}
The structure of the rest of this paper is as follows. In Section \ref{sec-variation}, we derive the hydrodynamics of liquid crystal elastomers \eqref{PHLC} by the energetic variational approach. In Section \ref{sec-3} we reformulate the system \eqref{ori_sys} by introducing new variables. Then we introduce the vector fields as well as define the Klainerman energy norms and weighted $L^2$ norms.  As preparations, in Section \ref{sec-dec} we show some weighted estimates and decay estimates.
Then, in Section \ref{sec-L2}, we  control the  higher and lower order  weighted $L^2$ norms by $E_N$ and $E_{N-2}$.

Section \ref{sec-HighEnergy} and \ref{sec-LowEnergy} are devoted to the proof of the  high-order energy estimate and low-order energy estimate, respectively. Here we shall use the  estimates established in  above sections and the cancellation structures of system.
Finally, in Section \ref{sec-2.3}, we prove the main Theorem \ref{Ori_thm} by  Proposition \ref{Prop_va} and \ref{LE-prop}. And hence the main Theorem \ref{Ori_thm0} is obtained by recovering map $d$ from angles $\phi$.

\bigskip
\section{Energetic variational approach}\label{sec-variation}
In this section, we derive the liquid crystal elastomers model \eqref{PHLC} by energic variational approach.

The energetic variational treatment of complex fluids starts with the energy dissipative law for the whole coupled system
\[  \frac{d}{dt}E^{tot}=-2\mathcal D,  \]
where $E^{tot}=E^{kinetic}+E^{int}$ is the total energy consisting of the kinetic energy and free energy. Here $\mathcal D$ is the dissipation function which is equal to the entropy production of the system in isothermal situations.
The inertial and conservative force from the kinetic and free
energies are, respectively, defined as
\begin{align}   \label{force-iner}
	\de \int_0^T E^{kinetic} dt=\int_0^T \< {\rm force}_{inertial},\de x  \>_{L^2_x} dt,\\\label{force-cons}
	\de \int_0^T E^{int} dt=\int_0^T \< {\rm force}_{conservative},\de x  \>_{L^2_x} dt,
\end{align}
Through ``Maximum Dissipation Principle", the dissipative force (linear with respect to the same rate function) can be derived as follows
\begin{align}   \label{force-dissi}
	\de D=\<{\rm force}_{dissipative},\de x_t\>.
\end{align}
When all forces are derived, the system follows according to the force balance
\begin{align}   \label{force-balance}
	{\rm force}_{inertial}={\rm force}_{conservative}+{\rm force}_{dissipative}.
\end{align}
For the above process in detail, we can also refer to Giga-Kirshtein-Liu \cite[Section 2.1]{GKL-2018}. The method is then used to derive the hydrodynamic of liquid crystal elastomers \eqref{PHLC}.

\medskip
In our system, we have the following dissipative law:
\begin{prop}
	If $(u,F,d)$ is a smooth solution to the system \eqref{PHLC}-\eqref{constraints1}, then the flow satisfies the following relation:
	\begin{equation}\label{DL-2.1}
		\begin{aligned}
			&\frac{1}{2}\frac{d}{dt}\Big( \|u\|_{L^2}^2+\|F\|_{L^2}^2+\rho_1\|\dot{ d}\|_{L^2}^2+\|\nab d\|_{L^2}^2  \Big)\\
			&=-\nu_1\|d^TAd\|_{L^2}^2-\frac{\nu_4}{2}\|\nab u\|_{L^2}^2
			+\la_1\left\Vert \dot{d}+Bd+\frac{\la_2}{\la_1}Ad\right\Vert_{L^2}^2-\Big(\nu_5+\nu_6+\frac{\la_2^2}{\la_1}\Big)\|Ad\|_{L^2}^2.
		\end{aligned}
	\end{equation}
\end{prop}
\begin{proof}
     From the system \eqref{PHLC}, we obtain
     \begin{align*}
     	&\frac{1}{2}\frac{d}{dt}\Big( \|u\|_{L^2}^2+\|F\|_{L^2}^2+\rho_1\|\dot{d}\|_{L^2}^2+\|\nab d\|_{L^2}^2  \Big)\\
     	&=\<u,\d_j (F_{\cdot k}F_{jk})-\d_j(\d_j d\cdot \nab d)+\div \sigma\>+\<F_{ik},\d_j u_i F_{jk}\>\\
     	&\quad +\<u\cdot\nab d,\De d \>+\<\dot{d},\la_1(\dot{d}+Bd)+\la_2 Ad\>,
     \end{align*}
 where we have used the constraints \eqref{constraints1} and $\dot{d}\cdot d=0$. By $\div u=0$, we compute
 \begin{align*}
 	&\<u,\d_j (F_{\cdot k}F_{jk})-\d_j(\d_j d\cdot \nab d)\>+\<F_{ik},\d_j u_i F_{jk}\>+\<u\cdot\nab d,\De d \>\\
 	&=-\<\d_j u_i,F_{ik}F_{jk}\>-\<u\cdot \nab d,\De d \>-\<u,\frac{1}{2}\nab|\nab d|^2\>+\<F_{ik},\d_j u_i F_{jk}\>+\<u\cdot\nab d,\De d \>\\
 	&=0.
 \end{align*}
Hence,
\begin{align*}
	\frac{1}{2}\frac{d}{dt}\Big( \|u\|_{L^2}^2+\|F\|_{L^2}^2+\rho_1\|\dot{d}\|_{L^2}^2+\|\nab d\|_{L^2}^2  \Big)
	=\<u,\div \sigma\>
	+\<\dot{d},\la_1(\dot{d}+Bd)+\la_2 Ad\>.
\end{align*}

Now we compute the terms on the right-hand side of the above equality, respectively. Precisely, by the expression of $\sigma$ in \eqref{Extra-Sress-sigma}, we calculate
\begin{align*}
	\<u_i,\d_j( \nu_1 d_k A_{kp}d_p  d_i d_j)\>&=-\nu_1\<d_k A_{kp}d_p  d_i d_j,\d_j u_i\>\\
	&= -\nu_1\<d_k A_{kp}d_p  d_i d_j,A_{ij}+B_{ij}\>=-\nu_1 \|d^T Ad\|_{L^2}^2,
\end{align*}
\begin{align*}
	&\<u_i,\d_j ( \nu_2  d_j N_i  + \nu_3 d_i N_j )\>\\
	&= -\<\d_j u_i, \nu_2  d_j (\dot{d}_i +B_{ki}d_k)  + \nu_3 d_i (\dot{d}_j+B_{kj}d_k) \>\\
	&=-\<A_{ij}+B_{ij}, (\nu_2-\nu_3)(d_j\dot{d}_i+d_j B_{ki}d_k)+\nu_3(d_j\dot{d}_i+d_i\dot{d}_j+d_j B_{ki}d_k+d_i B_{kj}d_k)\>\\
	&=-\la_1\<A_{ij},d_j\dot{d}_i+d_j B_{ki}d_k\>+\la_1\<B_{ji},d_j\dot{d}_i+d_j B_{ki}d_k\> -2\nu_3\<A_{ij},d_j\dot{d}_i+d_jB_{ki}d_k\>\\
	&=-(\la_1+2\nu_3)\<Ad,\dot{d}+Bd\>+\la_1\<Bd,\dot{d}+Bd\>\\
	&=\la_2\<Ad,\dot{d}+Bd\>+\la_1\<Bd,\dot{d}+Bd\>,
\end{align*}
and
\begin{align*}
	&\<u_i,\d_j(\nu_5 A_{ik}d_k d_j   + \nu_6 d_i A_{jk}d_k)\>\\
	&=-\<A_{ij}+B_{ij},(\nu_5-\nu_6) A_{ik}d_k d_j   + \nu_6 (A_{ik}d_k d_j+d_i A_{jk}d_k)\>\\
	&=-\la_2 \|Ad\|_{L^2}^2+\la_2\<Bd,Ad\>-2\nu_6\|Ad\|_{L^2}^2\\
	&=-(\nu_5+\nu_6)\|Ad\|_{L^2}^2+\la_2\<Bd,Ad\>.
\end{align*}
Collecting the above calculations, we obtain
\begin{align*}
	&\<u,\div \sigma\>
	+\<\dot{d},\la_1(\dot{d}+Bd)+\la_2 Ad\>\\
	&=-\nu_1 \|d^T Ad\|_{L^2}^2-\frac{\nu_4}{2}\|\nab u\|_{L^2}^2+\la_1\|\dot{d}+Bd\|_{L^2}+2\la_2\<Ad,\dot{d}+Bd\>-(\nu_5+\nu_6)\|Ad\|_{L^2}^2\\
	&=-\nu_1 \|d^T Ad\|_{L^2}^2-\frac{\nu_4}{2}\|\nab u\|_{L^2}^2+\la_1\left\Vert\dot{d}+Bd+\frac{\la_2}{\la_1}Ad\right\Vert_{L^2}-\Big(\nu_5+\nu_6+\frac{\la_2^2}{\la_1}\Big)\|Ad\|_{L^2}^2.
\end{align*}
Hence, the energy dissipative law \eqref{DL-2.1} follows.
\end{proof}

With the above energy dissipative law, the kinetic energy $E^{kinetic}$ and internal energy $E^{int}$ are given by
\begin{align}   \label{Ek-Ei}
	E^{kinetic}=\frac{1}{2}\|u\|_{L^2}^2+\frac{1}{2}\rho_1\|D_t  d\|_{L^2}^2,\qquad
	E^{int}=\frac{1}{2}\|F\|_{L^2}^2+\frac{1}{2}\|\nab d\|_{L^2}^2,
\end{align}
and the dissipation functional $\mathcal D$ is
\begin{align*}
	2\mathcal D=\nu_1\|d^T Ad\|_{L^2}^2+\frac{\nu_4}{2}\|\nab u\|_{L^2}^2-\la_1\left\Vert \dot{d}+ Bd +\frac{\la_2}{\la_1}Ad\right\Vert_{L^2}^2+\Big(\nu_5+\nu_6+\frac{\la_2^2}{\la_1}\Big)\|Ad\|_{L^2}^2.
\end{align*}
Next, we will use the variational method to derive the LCEs \eqref{PHLC} in several steps.

\medskip
At the beginning, we introduce some notations in order to rewrite the energies and dissipation functional in the Lagrangian coordinates.
Let $x(X,t)$ be the flow map with divergence free $\nab_x\!\cdot\! u=\nab_x \!\cdot\! \d_t x=0$. Let the map $\td(X,t):\R^n\times \R^+\rightarrow (\MM,h)\hookrightarrow \R^K$ be orientation field and $\tF(X,t)_{ij}=\frac{\d x_i}{\d X_j}$ be deformation gradient. We also denote $d(x,t)=\td(X(x,t),t)$ and $F(x(X,t),t)=\tF(X,t)$ in the Eulerian coordinates. The wave map equations can be written
on a manifold $(\R^n,g)$, where the metric
\begin{equation}
g_{ij}=\d_{X_i}x\cdot \d_{X_j}x
\end{equation}
are determined by the fluid equation of state.

From \eqref{Ek-Ei} and the above notations, we have the following action functional, in terms of the flow map $x(X,t)$ and the map $\td(X,t)$, in the Lagrangian coordinates
\begin{align*}
	\mathcal A&=\int_0^T (E^{kinetic}-E^{int}) \ dt\\
	&=\frac{1}{2}\int \big( |\d_t x|^2-|\nab x|^2 \big)\det \tF\  dXdt+\frac{1}{2}\int \big(\rho_1|\d_t \td|^2 -g^{ij}\d_i \td \d_j \td \big) \det \tF\ dX dt.
\end{align*}
The $A_{ij}$ and $B_{ij}$ are expressed as
\begin{align*}
	&\tilde A_{ij}(X,t)=\frac{1}{2}\Big(\frac{\d X_m}{\d x_i}\d_{X_m} \de \d_t x_j+\frac{\d X_n}{\d x_j}\d_{X_n} \de \d_t x_i\Big) ,\\
	&\tilde B_{ij}(X,t)=\frac{1}{2}\Big(\frac{\d X_m}{\d x_j}\d_{X_m} \de \d_t x_i-\frac{\d X_n}{\d x_i}\d_{X_n} \de \d_t x_j\Big).
\end{align*}

Denote
\[\frac{d}{d\ep}\Big|_{\ep=0}x^\ep=\de x,\quad  \frac{d}{d\ep}\Big|_{\ep=0}\td^\ep=\de \td \in T{\MM}.\]
Then we have
\begin{align*}
	\frac{d}{d\ep}\Big|_{\ep=0}\det \tF= \tr(\tF^{-1}\nab_X \de x)\det \tF,
\end{align*}
and
\begin{align}     \label{dg}
	\frac{d}{d\ep}\Big|_{\ep=0} g_{ij}=\d_{X_i} \de x\cdot \d_{X_j} x+\d_{X_i} x\cdot \d_{X_j} \de x.
\end{align}
From $g^{ij}g_{jk}=\de^i_k$, the formula \eqref{dg} also gives
\begin{align*}
	\frac{d}{d\ep}\bigg|_{\ep=0} g^{ij} =-g^{ik}g^{jl}\frac{d}{d\ep}\bigg|_{\ep=0} g_{kl}=-g^{ik}g^{jl}( \d_k \de x\cdot \d_l x+\d_k x\cdot \d_l \de x).
\end{align*}

\emph{Step 1. Derive the $F$-equation in \eqref{PHLC}}. By the chain rule, we deduce
\begin{align*}
	&\d_t F_{ij}+u\cdot \nab F_{ij}=\d_t (F(x(X,t),t)_{ij})=\d_t \tF_{ij}(X,t)=\d_t\frac{\d x_i}{\d X_j}\\
	&=\frac{\d}{\d X_j} u_i=\frac{\d x_k}{\d X_j}\frac{\d}{\d x_k} u_i=\d_k u_i F_{kj}=(\nab u F)_{ij}.
\end{align*}
That is
\begin{align}   \label{F}
	\d_t F+u\cdot \nab F=\nab u F,
\end{align}
as a byproduct, which combined with $\div u=0$ also implies that
\begin{align}       \label{tr}
	\tr (F^{-1}D_t F)=F^{ij} D_t F_{ji}=F^{ij} \d_ku_j F_{ki}=\de_k^j \d_k u_j=\d_j u_j=0.
\end{align}

\bigskip
\emph{Step 2. Derive the following forces with repect to $\de x$ and $\de d$}:
\begin{align}  \label{u-eq}
	&({\rm force}_{inertial}-{\rm force}_{conservative})_{\de x}=-\big(D_t u+\nab p-\nab\cdot (FF^T)+\div (\nab d\odot \nab d)\big),\\    \label{d}
	&({\rm force}_{inertial}-{\rm force}_{conservative})_{\de d}=-\rho_1 D^2_t d +\De d+(-\rho_1|D_t d|^2+|\nab d|^2)d.
\end{align}

Applying the operator $\frac{d}{d\ep}\big|_{\ep=0}$ to $\AA$.
For the first integral in $\AA$, we arrive at
\begin{align}   \label{FI-AA}
	&\frac{1}{2}\frac{d}{d\ep}\bigg|_{\ep=0}\int ( |\d_t x|^2-|\nab x|^2) \det \tF\ dXdt\\\nonumber
	&=\int  (\d_t x \d_t\de x-\nab x \nab \de x) \det \tF+\frac{1}{2}( |\d_t x|^2-|\nab x|^2) \tr(\tF^{-1}\nab_X\de x) \det \tF\ dXdt\\\nonumber
	&=\int  -(\d^2_t x-\De x ) \de x \det \tF-\big(\d_t x\tr(\tF^{-1}\d_t \tF)-\d_j x \tr(\tF^{-1}\d_j \tF)\big)\de x\det \tF\ dXdt\\\nonumber
	&\quad +\int \frac{1}{2} (|u|^2-|F|^2) \nab_x\cdot \de x\  dxdt\\\nonumber
	&=\int  -(D_t u_i-\d_l F_{ij} F_{lj} ) \de x_i -\big(u\tr(F^{-1}D_t F)-F_{ij} F^{kl}\d_j F_{lk}\big)\de x_i\ dxdt\\\nonumber
	&\quad +\int \frac{1}{2} (|u|^2-|F|^2) \nab\cdot \de x\  dxdt.
\end{align}
Thanks to \eqref{tr} and
\begin{align*}
	-F_{ij}F^{kl}\d_j F_{lk}=-F_{ij} \frac{\d X_k}{\d x_l}\d_{X_k}F_{lj}=-F_{ij}\d_{x_l}F_{lj},
\end{align*}
then we infer
\begin{align*}
	\eqref{FI-AA}&=-\int \Big((D_t u -\d_{x_l}(F_{\cdot j}F_{lj}))+\frac{1}{2}\nab  (|u|^2-|F|^2) \Big) \de x\ dxdt.
\end{align*}

For the second integral in $\AA$, we have
\begin{align*}
	&\frac{1}{2}\frac{d}{d\ep}\bigg|_{\ep=0}\int ( \rho_1|\d_t \td|^2 -g^{ij}\d_i \td \d_j \td ) \sqrt{\det g}\ dXdt\\
	&= \int \big( \rho_1\d_t \td  \d_t \de \td-\frac{1}{2}\de g^{ij} \d_i \td \d_j \td-g^{ij}\d_i \td \d_j \de \td \big)\sqrt{\det g}\\
	&\quad +\frac{1}{2} \big( \rho_1|\d_t \td|^2 -g^{ij}\d_i \td \d_j \td \big) \det \tF \tr(\tF^{-1}\nab \de x)\ dXdt\\
	&=-\int \Big( \rho_1\d^2_t \td  +\d_t \td \tr (\tF^{-1}\d_t \tF) -\frac{1}{\sqrt{\det g}}\d_j (\sqrt{\det g} g^{ij}\d_i \td)\Big) \de \td \sqrt{\det g}\ dXdt\\
	&\quad -\int \frac{1}{2}\de g^{ij} \d_i \td \d_j \td \det \tF+\frac{1}{2} ( -\rho_1|\d_t \td|^2 +g^{ij}\d_i \td \d_j \td ) \det \tF \tr(\tF^{-1}\nab \de x)\ dXdt\\
	:&=I+II.
\end{align*}
By the change of variables and \eqref{tr}, $I$ is written as
\begin{align*}
	I&=-\int \big(\rho_1 D_t^2 d+D_t d \tr (F^{-1}D_t F)-\De d\big) \de d \ dxdt
	=-\int (\rho_1 D_t^2 d-\De d) \de d \ dxdt,
\end{align*}
and $II$ is written as
\begin{align*}
	&-\int \frac{1}{2}\de g^{ij}\d_i \td\d_j \td \det \tF+\frac{1}{2} \big( -\rho_1|\d_t \td|^2 +g^{ij}\d_i \td \d_j \td \big) \det \tF \tr(\tF^{-1}\nab \de x)\ dXdt\\
	&=\int g^{ik}g^{jl}( \d_k \de x\cdot \d_l x)\d_i \td\d_j \td \det \tF-\frac{1}{2} \big( -\rho_1|\d_t \td|^2 +g^{ij}\d_i \td \d_j \td \big) \det \tF \tr(\tF^{-1}\nab \de x)\ dXdt\\
	&=\int \d_{x_\mu} d\nab_x d \d_{x_\mu}\de x  -\frac{1}{2} ( -\rho_1|D_t d|^2 +|\nab d|^2 )  \nab_x \de x \  dx dt\\
	&=\int -\d_{x_\mu}(\d_{x_\mu} d\nab_x d )\de x   +\frac{1}{2}\nab_x ( -\rho_1|D_t d|^2 +|\nab d|^2 )  \de x \  dx dt.
\end{align*}
Then we obtain
\begin{align*}
	&\frac{1}{2}\frac{d}{d\ep}\bigg|_{\ep=0}\int \big( \rho_1|\d_t \td|^2 -g^{ij}\d_i \td \d_j \td \big) \sqrt{\det g}\ dXdt\\
	&=-\int (\rho_1 D_t^2 d-\De d) \de d \ dxdt-\int \d_{x_\mu}(\d_{x_\mu} d\nab_x d )\de x   -\frac{1}{2}\nab_x ( -\rho_1|D_t d|^2 +|\nab d|^2 )  \de x \  dx dt.
\end{align*}

Hence,
\begin{equation}\label{dA}
	\begin{aligned}
		\frac{d}{d\ep}\bigg|_{\ep=0}\AA=&\ -\int \Big(D_t u -\nab\cdot(FF^T)+\frac{1}{2}\nab  (|u|^2-|F|^2+|D_t d|^2-|\nab d|^2)\\
		&+\d_j(\nab d \d_j d) \Big) \de x\ dxdt
		-\int (\rho_1 D_t^2 d-\De d) \de d \ dxdt.
	\end{aligned}
\end{equation}
From the definitions of ``forces" in \eqref{force-iner} and \eqref{force-cons}, this yields the relation in term of $d$,
\begin{align}   \label{d-general}
	({\rm force}_{inertial}-{\rm force}_{conservative})_{\de d}=-\rho_1D^2_t d +\De d+S_{bc}(d)(-\rho_1D_t d^b D_t d^c+\nab d^b \nab d^c),
\end{align}
where $d(x)\in (\MM,h)\hookrightarrow \R^K$ and $S$ is the second
fundamental form of $\MM$, viewed as a symmetric bilinear form $S: T\MM\times T\MM\rightarrow N\MM$.
Particularly, when the target manifold is a sphere $\S^2$, the formula \eqref{d-general} is reduced to \eqref{d}.
At the same time, under the constriant $\div u=0$, the formula \eqref{dA} combined with the standard method in \cite[P499, Theorem 6]{Evans} implies the existence of a scalar function $\tp$ such that
\begin{align*}
	({\rm force}_{inertial}-{\rm force}_{conservative})_{\de x}=-\big(D_t u+\nab p-\nab\cdot (FF^T)+\div (\nab d\odot \nab d)\big).
\end{align*}
Thus the formula \eqref{u-eq} is also obtained.

\bigskip
\emph{Step 3. Derive the ``$\ {\rm force}_{dissipative}$" with respect to $\de \d_tx$ and $\de \d_t \tilde d$:}
\begin{align}   \label{Var-v}
	&({\rm force}_{dissipative})_{\de \d_t x}=\nab p-\nab\cdot\si,\\ \label{Var-d}
	&({\rm force}_{dissipative})_{\de \d_t\tilde d}=-\la_1 (D_t d+Bd)-\la_2 Ad.
\end{align}

According to the maximum dissipation principle, we take $\de_{\d_t x}\mathcal D$ with incompressibility of the fluid $\nab\cdot u=0$.
\begin{align}  \label{de-la1}
	&\de_{\d_t x}\Big(-\frac{1}{2}\la_1\left\Vert \d_t \tilde d+ \tilde B\tilde d +\frac{\la_2}{\la_1}\tilde A\tilde d\right\Vert_{L^2}^2\Big)\\\nonumber
	&=-\la_1\int \Big( \d_t \tilde d_j+ (\tilde B\tilde d)_j +\frac{\la_2}{\la_1}(\tilde A\tilde d)_j\Big) \Big( \frac{1}{2}\big(\frac{\d X_m}{\d x_j}\d_{X_m} \de \d_t x_k-\frac{\d X_n}{\d x_k}\d_{X_n} \de \d_t x_j\big) \tilde d_k \\\nonumber
	&\quad +\frac{\la_2}{\la_1}\frac{1}{2}\big(\frac{\d X_m}{\d x_k}\d_{X_m} \de \d_t x_j+\frac{\d X_n}{\d x_j}\d_{X_n} \de \d_t x_k\big) \tilde d_k   \Big)\ dX\\\nonumber
	&=-\la_1\int \Big( \d_t \tilde d_j+ (\tilde B\tilde d)_j +\frac{\la_2}{\la_1}(\tilde A\tilde d)_j\Big) \Big( \frac{1}{2}(1+\frac{\la_2}{\la_1})\frac{\d X_m}{\d x_j}\d_{X_m} \de \d_t x_k d_k\\\nonumber
	&\quad -\frac{1}{2}(1-\frac{\la_2}{\la_1})\frac{\d X_n}{\d x_k}\d_{X_n} \de \d_t x_j d_k\Big) \ dX.
\end{align}
Again, in the Eulerian coordinates, the above integral is expressed as
\begin{align*}
	\eqref{de-la1}&=-\la_1\int \Big( \dot{d}_j+ ( B d)_j +\frac{\la_2}{\la_1}( A d)_j\Big) \Big( \frac{1}{2}(1+\frac{\la_2}{\la_1})\d_j(\de \d_t x)_k  d_k\\
	&\quad -\frac{1}{2}(1-\frac{\la_2}{\la_1})\d_k (\de \d_t x)_j d_k\Big) \ dx\\
	&=\la_1\int \frac{1}{2}(1+\frac{\la_2} {\la_1})\d_j\Big( N_j d_k +\frac{\la_2}{\la_1}(Ad)_jd_k\Big)(\de \d_t x)_k \\
	&\quad -\frac{1}{2}(1-\frac{\la_2}{\la_1})\d_j\Big( N_k d_j +\frac{\la_2}{\la_1}(Ad)_kd_j\Big) (\de \d_t x)_k  \ dx\\
	&=\int \Big(-\nu_2\d_j(d_j N_k)-\nu_3\d_j(d_kN_j)+\frac{1}{2}(\la_2+\frac{\la_2^2}{\la_1})\d_j (A_{jm}d_md_k)\\
	&\quad -\frac{1}{2}(\la_2-\frac{\la_2^2}{\la_1})\d_j(A_{km}d_md_j)\Big) (\de \d_t x)_k  \ dx.
\end{align*}
For the other terms, we also have
\begin{align*}
	&\de_{\d_t x}\Big(\frac{1}{2}\nu_1\|d^T Ad\|_{L^2}^2\Big)=\de_{u}\Big(\frac{1}{2}\nu_1\|d^T Ad\|_{L^2}^2\Big)\\
	&=\nu_1\int d_k A_{kp}d_pd_i \frac{1}{2}(\d_i \de u_j+\d_j \de u_i)d_j\ dx\\
	&= -\nu_1\int \d_j(d_k A_{kp}d_pd_id_j) \de u_i\ dx,
\end{align*}
\begin{align*}
	&\de_{\d_t x}\Big(\frac{\nu_4}{4}\|\nab u\|_{L^2}^2\Big)=\de_{u}\Big(\frac{\nu_4}{4}\|\nab u\|_{L^2}^2\Big)\\
	&=\int \frac{\nu_4}{2} \nab u \nab \de  u dx=-\int \frac{\nu_4}{2} \De u \ \de  u dx=-\int \nu_4 \d_j A_{ij}  \ \de  u_i dx,
\end{align*}
and
\begin{align*}
	&\de_{\d_t x}(\frac{1}{2}\Big(\nu_5+\nu_6+\frac{\la_2^2}{\la_1}\Big)\|Ad\|_{L^2}^2)=\de_{u}(\frac{1}{2}\Big(\nu_5+\nu_6+\frac{\la_2^2}{\la_1}\Big)\|Ad\|_{L^2}^2)\\
	&=\Big(\nu_5+\nu_6+\frac{\la_2^2}{\la_1}\Big)\int A_{ki}d_k \frac{1}{2} (\d_j \de u_i+\d_i \de u_j) d_j \  dx\\
	&=- \Big(\nu_5+\nu_6+\frac{\la_2^2}{\la_1}\Big)\int \frac{1}{2}\d_j(A_{ki}d_k d_j)   \de u_i+\frac{1}{2}\d_j(A_{kj}d_k d_i)  \de u_i \  dx.
\end{align*}
In view of the relation $\la_2=\nu_5-\nu_6$, collecting the above calculations yield
\begin{align*}
	\de_{\d_t x}\DD = \int -\div \si\cdot \de u\ dx.
\end{align*}
And hence the dissipative force \eqref{Var-v} follows under the constraint $\div \de u=0$.

Similarly, since the third term in $\mathcal D$ related to $\d_t \tilde d$, we take $\de_{\d_t \tilde d}\mathcal D$ to deduce
\begin{align*}
	\de_{\d_t \tilde d}\mathcal D&=\de_{\d_t \tilde d}\Big(-\frac{1}{2}\la_1\left\Vert \d_t \tilde d+ \tilde B\tilde d +\frac{\la_2}{\la_1}\tilde A\tilde d\right\Vert_{L^2}^2\Big)\\
	&=-\la_1\int \Big(\d_t \tilde d+ \tilde B\tilde d +\frac{\la_2}{\la_1}\tilde A\tilde d\Big) \de \d_t \tilde d\ dX\\
	&=-\int (\la_1 (\d_t \tilde d+ \tilde B\tilde d) +\la_2\tilde A\tilde d)\ \de \d_t \tilde d\ dX,
\end{align*}
which yields the formula \eqref{Var-d} in Eulerian coordinates.

According to the force balance \eqref{force-balance}, from \eqref{u-eq} and \eqref{Var-v} we obtain the $u$-equations in \eqref{PHLC}, by \eqref{d} and \eqref{Var-d} we also get the $d$-equations in \eqref{PHLC}.
Therefore, the system \eqref{PHLC} is obtained.

\bigskip
\section{Reformulation and function spaces}\label{sec-3}
In this section, we reformulate the system \eqref{ori_sys}-\eqref{constraints1-re} by introducing some new variables. Here we also introduce the vector fields, which are used to define the weighted energy norms and $L^2$ weighted norms.
\subsection{Derivation of \eqref{EL-elas-1}}
Here we derive the system \eqref{EL-elas-1} from \eqref{ori_sys}. Recall the notation
\begin{equation*}
	H=F-I,
\end{equation*}
 the equation of $F$ is written as
\begin{equation}   \label{H}
	\d_t H-\nab u =-u\cdot \nab H+\nab u H,
\end{equation}
with constraints
\begin{align}    \label{constriant1-re}
	\d_j H_{ik}-\d_k H_{ij}=H_{mk}\d_m H_{ij}-H_{mj} \d_m H_{ik},\quad  \nab\cdot H^T=0.
\end{align}
The quadratic term $\nab\cdot (FF^{\top})$ is rewritten as
\begin{align*}
	\nab\cdot(FF^T)&=\nab\cdot[(I+H)(I+H)^T]=\nab\cdot H+\nab\cdot H^T+\nab\cdot(HH^T)\\
	&=\nab\cdot H+\nab\cdot(HH^T).
\end{align*}
Then the equation of $u$ is given by
\begin{equation}   \label{u}
	\d_t u-\nab\cdot H+\nab p=-u\cdot \nab u+\nab\cdot(HH^T)-\sum_{j=1}^3 \d_j  (\nab d\cdot \d_j d).
\end{equation}

In view of the expression
\begin{align*}
	d=(\cos \phi_1\cos\phi_2,\sin\phi_1\cos\phi_2,\sin\phi_2),
\end{align*}
  we know that the angles $\phi_1$, $\phi_2$ are near $0$, since the orientation $d$ is near $i$. It follows from  the third component $d_3$ that
\begin{align*}
	\cos\phi_2 (D_t^2\phi_2-\De\phi_2)+\sin \phi_2(-|D_t\phi_2|^2+|\nab\phi_2|^2)=&\cos^2 \phi_2(-|D_t\phi_1|^2+|\nab\phi_1|^2)\sin\phi_2\\
	&+(-|D_t\phi_2|^2+|\nab\phi_2|^2)\sin\phi_2,
\end{align*}
 which implies
\begin{align} \label{phi2}
	D_t^2\phi_2-\De\phi_2&=\sin\phi_2\cos \phi_2(-|D_t\phi_1|^2+|\nab\phi_1|^2).
\end{align}
From the first component $d_1$, we have
\begin{align*}
	&-\sin\phi_1\cos\phi_2 (D_t^2\phi_1-\De\phi_1)-\cos\phi_1\sin\phi_2 (D_t^2\phi_2-\De\phi_2)\\
	&-\cos\phi_1\cos\phi_2(|D_t\phi_1|^2-|\nab\phi_1|^2)+2\sin\phi_1\sin\phi_2(D_t\phi_1D_t\phi_2-\nab\phi_1\nab\phi_2)\\
	&=\cos\phi_1 \cos^3\phi_2 (-|D_t\phi_1|^2+|\nab\phi_1|^2).
\end{align*}
This combined with \eqref{phi2} gives
\begin{align*}
	(D_t^2\phi_1-\De\phi_1)
	=2\tan\phi_2(D_t\phi_1D_t\phi_2-\nab\phi_1\nab\phi_2).
\end{align*}
Hence the equation of $d$ can be rewritten as:
\begin{equation}   \label{phi}
	(\d_t+u\cdot \nab)^2\phi-\De \phi=\mathcal R_2,
\end{equation}	
where $\mathcal R_2$ in \eqref{R2} is the high order term.
We can also rearrange the quadratic term in \eqref{u} as
\begin{equation*}
	\sum_j\d_j (\nab d \cdot \d_j d )=\sum_j\d_j(\nab\phi\cdot \d_j\phi)-\mathcal R_1,
\end{equation*}
where $\mathcal R_1$ in \eqref{R1} is also a higher order term.

In conclusion, by \eqref{H}, \eqref{u}, \eqref{phi} and \eqref{constriant1-re}, we obtain the new formulation \eqref{EL-elas-1} with constriants \eqref{cst-re} and \eqref{curl}.

\subsection{Vector fields and function spaces}\label{sec-2.2}
Here we first introduce the vector fields, and then define the energy norms and weighted $L^2$ generalized energies. In the end, we state the main bootstrap proposition.

We defne the perturbed angular momentum operators by
\begin{align*}
	\tOm_i u=\Om_i u +A_i u,\quad \tOm_i H=\Om_i H +[A_i ,H], \quad \tOm\phi=\Om\phi,
\end{align*}
where $\Om = (\Om_1, \Om_2, \Om_3)$ is the rotation vector feld $\Om = x \times \nab $ and $A_i$ is defned by
\begin{equation*}
	A_1=\left(\begin{array}{ccc}
		0&0&0\\
		0&0&1\\
		0&-1&0
	\end{array}\right),\ \
	A_2=\left(\begin{array}{ccc}
		0&0&-1\\
		0&0&0\\
		1&0&0
	\end{array}\right),\ \
	A_3=\left(\begin{array}{ccc}
		0&1&0\\
		-1&0&0\\
		0&0&0
	\end{array}\right),
\end{equation*}
and $[A_i,H]=A_i H-HA_i$ denotes the standard Lie bracket product.
We define the scaling vector-field $S$ by
\begin{equation*}
	S=t\d_t+x_i\d_{x_i},
\end{equation*}
and the perturbed scaling operators as
\begin{equation*}
	\tS u=Su,\quad \tS H=SH,\quad \tS \phi=(S-1)\phi.
\end{equation*}

Let
\begin{equation*}
	Z=(Z_1,\cdots,Z_8)= \{\d_t,\d_1,\d_2,\d_3,\tOm_1,\tOm_2,\tOm_3,\tilde S\}.
\end{equation*}
For any $a=(a_1,\cdots,a_8)\in\Z_+^8$, we denote $Z^{a}=Z_1^{a_1}\cdots Z_8^{a_8}$.
We define the Klainerman?¡¥s generalized energy as
\begin{equation}   \label{def-Ea}
	E_{a}(u,H,\phi;t)=\frac{1}{2}\int |Z^a u|^2+|Z^a H|^2+|D_t Z^a\phi|^2+|\nab_{x} Z^a\phi|^2\ dx,
\end{equation}
where $D_t=\d_t+u\cdot \nab$ is the material derivative.
We also need the weighted $L^2$ generalized energy,
\begin{align*}    
	\XX_{a}(t)=\|\<t-r\>\nab Z^a u\|_{L^2}^2+\|\<t-r\>\nab Z^a H\|_{L^2}^2+ \|\<t-r\> |D^2 Z^a\phi\|_{L^2}^2 .
\end{align*}

However, the above generalized energy $E_a(t)$ does not guarantee the desired energy estimates. To deal with the nonlinear terms, we introduce the  the following modified energy functional
\begin{align*}
	\EE_a(t):&=E_a(t)+\int \frac{1}{2} |Z^a u\cdot\nab\phi|^2+D_t Z^a\phi (Z^au\cdot\nab\phi)\ dx,
\end{align*}
and
\begin{align*}
	\bE_a(t):&=\EE_a(t) -\frac{1}{2} \int \sin^2 \phi_2\big( |D_t Z^a \phi_1|^2+|\nab Z^a \phi_1|^2\big)\ dx\\
	&\quad -\int \frac{1}{2}\sin^2\phi_2|Z^au\cdot \nab \phi_1|^2+\sin^2\phi_2 D_t Z^a\phi_1 Z^a u\cdot\nab\phi_1\ dx.
\end{align*}
The above energy norms capture all of cancellations, and are equivalent to $E_a(t)$ in our setting. In fact,  we shall use the first modified energy $\EE_a$ to derive the lower-order energy estimate \eqref{LE}, and apply the second one $\bE_a$ to obtain the higher-order energy estimate \eqref{Ev}.

For sake of convenience, we denote
\begin{align}   \label{En}
	E_j=\sum_{|a|\leq j}E_a,\qquad \XX_j=\sum_{|a|\leq j}\XX_a,\qquad \EE_j=\sum_{|a|\leq j}\EE_a,\qquad \bE_j=\sum_{|a|\leq j}\bE_a,
\end{align}
and in a function space $U$, for any $j\in \Z^+$, we denote
\begin{align*}
	\|Z^j f\|_{U}=\sum_{|a|\leq j}\|Z^a f\|_{U}.
\end{align*}

In order to characterize the initial data, we introduce the time independent
analogue of $Z$. The only difference will be in the scaling operator. Set
\begin{equation*}
	\Lambda=(\Lambda_1,\cdots,\Lambda_7)=(\d_1,\d_2,\d_3,\tilde \Om_1,\tilde \Om_2,\tilde \Om_3,\tilde S_0),\quad \tilde S_0=\tilde S-t\d_t.
\end{equation*}
Then the commutator of any two $\Lambda$'s is again a $\Lambda$. Define
\begin{equation}   \label{def-HLam}
	H^m_\Lambda=\Big\{ (u,H,\psi_0,\psi_1):\sum_{|a|\leq m}\big(\|\Lambda^a (u,H)\|_{L^2}+\|\nab \Lambda^a \psi_0\|_{L^2} +\| \Lambda^a \psi_1\|_{L^2}\big) <\infty \Big\}.
\end{equation}
We shall solve the liquid crystal elastomers in the space
\begin{align*}
	H^m_Z(T)=\Big\{ (u,H,\phi): Z^a u,Z^a H, \d_tZ^a\phi,\nab Z^a\phi\in L^\infty([0,T]; L^2) ,\ \forall |a|\leq m      \Big\}.
\end{align*}

Applying the vector fields to \eqref{EL-elas-1}, we can derive that
\begin{equation}     \label{EL-VF}
	\left\{
	\begin{aligned}
		&\d_t Z^a u-\nab\cdot Z^a H+\nab Z^a p=f_a,\\
		&\d_t Z^a H-\nab Z^a u = g_a,\\
		&D_t^2Z^a\phi-\Delta Z^a\phi=h_a,
	\end{aligned}
	\right.
\end{equation}
with constraints
\begin{align}  \nonumber
	&\div Z^a u=0,\quad \nab\cdot Z^a H^T=0, \\ \label{curl-vf}
	&\d_j Z^a H_{ik}-\d_k Z^a H_{ij}=\mathcal N_{a,jik},
\end{align}
where the nonlinearities are defined as
\begin{align}
    &\begin{aligned}\label{fa}
    f_a:&=-\sum_{b+c=a}C_a^b Z^b u\cdot \nab Z^c u +\sum_{b+c=a}C_a^b\nab\cdot (Z^b H Z^c H^T)\\
	&\quad - \sum_{b+c=a}C_a^b\d_j  (\nab Z^b\phi\cdot \d_j Z^c\phi)+\RR_{1;a},
    \end{aligned}	\\   \label{ga}
	&g_a:=-\sum_{b+c=a}C_a^b Z^b u\cdot \nab Z^cH+\sum_{b+c=a}C_a^b \nab Z^b u Z^c H,
\end{align}
\begin{align}   \label{ha}
	&\begin{aligned}
		h_a:&= -D_t(Z^au\cdot\nab\phi)-\sum_{b+c=a;b,c\neq a}C_a^b\d_t (Z^b u\cdot\nab Z^c\phi)-\sum_{b+c=a;c\neq a}C_a^bZ^b u\cdot\nab\d_tZ^c\phi\\
		&\quad -\sum_{b+c+e=a;c,e\neq a}C_a^{b,c}Z^bu\cdot\nab(Z^cu\cdot\nab Z^e\phi)+\mathcal R_{2;a},
	\end{aligned}\\   \label{Na}
	&\mathcal N_{a,jik}:=\sum_{b+c=a}\Big(C_a^b Z^b H_{mk}\d_m Z^c H_{ij}- Z^b H_{mj} \d_m  Z^c H_{ik}\Big).
\end{align}
Here the error terms $\RR_{1;a}$ and $\RR_{2;a}$ are given by
\begin{align}   \label{R1a}
	&\RR_{1;a}=\sum_{b+c+e=a}C_{a}^{b,c}\sum_j\d_j(Z^b\sin^2\phi_2 \nab Z^c\phi_1 \d_j Z^e\phi_1),\\
	&\mathcal R_{2;a}=\sum_{b+c+e=a}C_a^{b,c}
	\left(\begin{aligned}
		2Z^b\tan\phi_2(D_t Z^c\phi_1 D_t Z^e\phi_2-\nab Z^c\phi_1\nab Z^e\phi_2)\\
		\frac{1}{2}Z^b\sin 2\phi_2(-D_t Z^c\phi_1 D_t Z^e\phi_1+\nab Z^c\phi_1 \nab Z^e\phi_1)
	\end{aligned}\right).
\end{align}
Note that
\begin{equation*}
	\tilde S (\sin^2\phi_2):= S (\sin^2\phi_2),\quad \tilde S(\tan\phi_2):=(S-1)\tan\phi_2, \quad \tilde S(\sin 2\phi_2):=(S-1)\sin 2\phi_2,
\end{equation*}
and the constant coefficients
\begin{equation*}
	C_a^b:=\frac{a!}{b!(a-b)!},\qquad  C_a^{b,c}:=\frac{a!}{b!c!(a-b-c)!}.
\end{equation*}

\bigskip
\section{Weighted estimates and decay estimates}\label{sec-dec}

In this section, we provide some weighted estimates and decay estimates, which will be frequently used in the later sections.
Here we start with the following weighted $L^{\infty}\mbox{-}L^2$ estimates.
\begin{lemma}
	Let $f\in H^2(\R^3)$, then there hold
	\begin{align}  \label{r1/2u}
		&\<r\>^{1/2}|f(x)|\lesssim \sum_{|\al|\leq 1}\|\nab\Om^{\al}f\|_{L^2},\\\label{ru}
		&\<r\>|f(x)|\lesssim \sum_{|\al|\leq 1}\|\d_r\Om^{\al}f\|_{L^2}^{1/2}\sum_{|\al|\leq 2}\|\Om^{\al}f\|_{L^2}^{1/2}.
	\end{align}
	In particular, let $\om=x/|x|$, assume that $\div g=0$ for a vector $g=(g_1,g_2,g_3)$, then
	\begin{equation}  \label{omf}
		\|r^{3/2}(\om\cdot g)\|_{L^\infty}\lesssim \sum_{|\alpha|\leq 2}\lV \Om^{\alpha}g\rV_{L^2}.
	\end{equation}
\end{lemma}
\begin{proof}
	For \eqref{r1/2u} and \eqref{ru}, one please refers to Lemma 4.2 in \cite{KS96} and Lemma 3.3 in \cite{Si}.
	Here we are aimed at  proving the bound \eqref{omf}.
	From \eqref{r1/2u} and the decomposition
	\begin{equation}    \label{decomposition}
		\nab=\om\d_r-\frac{\om}{r}\times\Om,\qquad \om=x/|x|.
	\end{equation}
    we have
	\begin{align}   \nonumber
		\|r^{3/2}(\om\cdot g)\|_{L^\infty}&\les \sum_{|a|\leq 1 }\|\nab\Om^a (r\om\cdot g)\|_{L^2}\les \sum_{|a|\leq 1 }\|\nab (r\om\cdot \tilde\Om^a g)\|_{L^2}\\\nonumber
		&\les  \sum_{|a|\leq 1}\|\d_r (r\om\cdot \tilde\Om^a g)\|_{L^2}+\sum_{|a|\leq 1}\|\frac{x}{r^2}\times \Om (r\om\cdot \Om^a g)\|_{L^2}\\ \label{omg}
		&\les \sum_{|a|\leq 1}\| r\om\cdot \d_r\tilde\Om^a g\|_{L^2}+\sum_{|a|\leq 2}\| \Om^a g\|_{L^2}.
	\end{align}
Since $\div g=0$, with the help of \eqref{decomposition} again, we derive
\begin{equation*}
	0=\div g=\om\cdot \d_r g-\frac{\om}{r}\times\Om \cdot g=\d_r(\om \cdot g)-\frac{\om}{r}\times\Om \cdot g.
\end{equation*}
The first term in the right hand side of \eqref{omg} can be controlled by
\begin{align*}
	\sum_{|a|\leq 1}\| r\om\cdot \d_r\tilde\Om^a g\|_{L^2}=\sum_{|a|\leq 1}\| \om\times \Om\cdot\tilde\Om^a g\|_{L^2}\les \sum_{|a|\leq 2}\| \tilde \Om^a g\|_{L^2}.
\end{align*}
Thus the bound \eqref{omf} is obtained.
\end{proof}

\begin{lemma} \label{Decay-phi}
	Assume that $(u,H,\phi)$ is the solution of \eqref{EL-elas-1}. Then there hold
	\begin{align}  \label{AwayCone}
		&\<r\>| Z^a u|+\<r\>| Z^a H|+\<r\>|\nab_{t,x} Z^a \phi|\lesssim E^{1/2}_{|a|+2},\\  \label{omuH}
		&\<r\>^{3/2}|\om\cdot Z^a u|+\<r\>^{3/2}|\om\cdot Z^a H^T|\les E^{1/2}_{|a|+2}.
	\end{align}
\end{lemma}
\begin{proof}
	The first bound follows from \eqref{ru}. The second one is obtained by $\div Z^au=0$, $\nab\cdot Z^a H^T=0$ and the bound \eqref{omf}.
\end{proof}

Next we state some decay estimates.
\begin{lemma}  \label{lem4.3}
	For all $f\in H^2(\R^3)$, there hold
	\begin{align} \label{decay-1}
		&\<t\>\|f\|_{\lf(r<2\<t\>/3)}\les \|f\|_{L^2}+\|\<t-r\>\nab f\|_{L^2}+\|\<t-r\>\nab^2 f\|_{L^2},\\  \label{decay-2}
		&\<t\>\|f\|_{L^6(r<2\<t\>/3)}\les \|f\|_{L^2}+\|\<t-r\>\nab f\|_{L^2}.
	\end{align}
\end{lemma}
\begin{proof}
	Let $\varphi\in C_0^\infty$, satisfy $\varphi(s)=1$ for $s\leq 1$, $\varphi(s)=0$ for $s\geq 2$. By the Sobolev embedding theorem
	\begin{equation*}
		\|f\|_{\lf}\les \|\nab f\|_{L^2}^{1/2}\|\nab^2 f\|_{L^2}^{1/2},
	\end{equation*}
we have
\begin{align*}
	\<t\>\big|\varphi(\frac{r}{2\<t\>/3})f\big|
	&\les \<t\>\Big\|\nab(\varphi(\frac{r}{2\<t\>/3})f)\Big\|_{L^2}+\<t\>
    \Big\|\nab^2(\varphi(\frac{r}{2\<t\>/3})f)\Big\|_{L^2}\\
	&\les \|f\|_{L^2}+\<t\>\|\nab f\|_{L^2(r\leq 4\<t\>/3)}+\<t\>\|\nab^2 f\|_{L^2(r\leq 4\<t\>/3)}\\
	&\les \|f\|_{L^2}+\|\<t-r\>\nab f\|_{L^2}+\|\<t-r\>\nab^2 f\|_{L^2}.
\end{align*}

For the second bound \eqref{decay-2}, by the Sobolev embedding theorem, we derive
\begin{align*}
	\<t\>\|f\|_{L^6(r<2\<t\>/3)}&\les \Big\|\<t\>\varphi(\frac{r}{2\<t\>/3})f\Big\|_{L^6(r<2\<t\>/3)}\\
	&\les \Big\|\<t\>\nab(\varphi(\frac{r}{2\<t\>/3})f)\Big\|_{L^2(r<2\<t\>/3)}\\
	&\les \|f\|_{L^2}+\|\<t-r\>\nab f\|_{L^2}.
\end{align*}
\end{proof}

From Lemma \ref{Decay-phi} and \ref{lem4.3}, we have the following decay estimates.
\begin{prop}
	Let $0<\de\leq 1/8$ be as in Theorem \ref{Ori_thm}. Assume that $(u,H,\phi)$ is a solution of \eqref{EL-elas-1}. There hold
	\begin{align}      \label{Decay-phi2}
		&\|Z^a u\|_{\lf}+\| Z^a H\|_{\lf}+\|\nab Z^a \phi\|_{\lf}\les \<t\>^{-1}(E^{1/2}_{|a|+2}+\XX^{1/2}_{|a|+2}),\\  \label{Zaphi}
		&\|Z^a\phi\|_{\lf}\les \<t\>^{-1/3+\de}(E^{1/2}_{|a|+2}+\XX^{1/2}_{|a|+2}).
	\end{align}
\end{prop}
\begin{proof}
	Here we first consider the estimate of  $Z^a u$. Precisely, we  have from  \eqref{ru} and \eqref{decay-1} that
	\begin{align*}
		\<t\>\|Z^a u\|_{\lf}& \les\<t\>\|Z^a u\|_{\lf(r<2\<t\>/3)}+\<t\>\|Z^a u\|_{\lf(r\geq 2\<t\>/3)}\\
		&\les \|Z^a u\|_{L^2}+\|\<t-r\>\nab Z^{|a|+1}u\|_{L^2}+\sum_{|b|\leq 1}\|\d_r \Om^b u\|_{L^2}^{1/2}\sum_{|b|\leq 2}\|\Om^b Z^au\|_{L^2}^{1/2}\\
		&\les E^{1/2}_{|a|+2}+\XX^{1/2}_{|a|+2}.
	\end{align*}
The other terms can be obtained in a similar way. Then the desired bound \eqref{Decay-phi2} follows.

Next, we denote $P_{\leq 0}$, $P_{>0}$ as the projection of low frequency and high frequency, respectively. It follows from  Bernstein's inequality that
\begin{align*}
	\|Z^a\phi\|_{\lf}&\les \|\nab P_{\leq 0}Z^a\phi\|_{L^{3-\de_1}}+\|\nab P_{>0} Z^a\phi\|_{\lf}\\
	&\les \|\nab P_{\leq 0}Z^a\phi\|_{L^2}^{2/(3-\de_1)}\|\nab P_{\leq 0}Z^a\phi\|_{\lf}^{1-2/(3-\de_1)}+\|\nab P_{>0} Z^a\phi\|_{\lf}.
\end{align*}
Combining \eqref{Decay-phi2}, and choosing $2/(3-\de_1)=2/3+\de$, we arrive at
\begin{align*}
	\|Z^a\phi\|_{\lf}&\les E^{1/(3-\de_1)}_{|a|}\<t\>^{-1+2/(3-\de_1)}(E^{1/2}_{|a|+2}+\XX^{1/2}_{|a|+2})^{1-2/(3-\de_1)}+\<t\>^{-1}(E^{1/2}_{|a|+2}+\XX^{1/2}_{|a|+2})\\
	&\les \<t\>^{-1/3+\de}(E^{1/2}_{|a|+2}+\XX^{1/2}_{|a|+2}).
\end{align*}
Thus the bound \eqref{Zaphi} follows.
\end{proof}

As a corollary, we have the following estimates.
\begin{cor}
	Let $0<\de\leq 1/8$ be as in Theorem \ref{Ori_thm}. Assume that $(u,H,\phi)$ is a solution of \eqref{EL-elas-1} with $E^{1/2}_{N-2}\les \ep$, then there hold
	\begin{align}     \label{sin1}
		&\|Z^a\sin\phi\|_{L^6}\les \|\nab Z^{|a|}\phi\|_{L^2}, \qquad\qquad\qquad\ \  |a|\leq N,\\   \label{sin2}
		&\|Z^a\sin\phi\|_{\lf}\les  \<t\>^{-1/3+\de}(E^{1/2}_{|a|+2}+\XX^{1/2}_{|a|+2}),\quad |a|\leq N-2.
	\end{align}
\end{cor}
\begin{proof}
	By the series expansion
	\begin{align*}
		\sin\phi=\phi-\int_0^\phi \sin t (\phi-t)\ dt,
	\end{align*}
we have
\begin{align}      \label{sinphi}
	|Z^a\sin\phi|\les |Z^a\phi|+ \Big|\int_0^\phi \sin t Z^a(\phi-t) dt\Big|+ \sum_{i\leq |a|;|a_1|+\cdots |a_i|\leq |a|} |Z^{a_1}\phi \cdots Z^{a_i}\phi|.
\end{align}
This yields the estimate
\begin{align*}
	\|Z^a\sin\phi\|_{L^6}\les \|\nab Z^{|a|}\phi\|_{L^2}(1+\|\phi\|_{\lf}+\|\nab Z^{N-2}\phi\|_{L^2})\les \|\nab Z^{|a|}\phi\|_{L^2}.
\end{align*}
The formula \eqref{sinphi} combined with \eqref{Zaphi} also gives
\begin{align*}
	\|Z^a\sin\phi\|_{\lf}\les \| Z^{|a|}\phi\|_{\lf}(1+\|\phi\|_{\lf}+\|\nab Z^{N-2}\phi\|_{L^2})\les \<t\>^{-1/3+\de}(E^{1/2}_{|a|+2}+\XX^{1/2}_{|a|+2}).
\end{align*}
\end{proof}

\bigskip
\section{Estimates of the \texorpdfstring{$L^2$}{} Weighted Norm}\label{sec-L2}
In this section, we focus on the estimates of $L^2$ weighted norm $\XX_{j}$. For clarification of notations, we denote
\begin{align*}
	&\XX^{uH}_{j}(t):=\sum_{|\al|\leq j-1}\Big(\|\<t-r\>\nab Z^a u\|_{L^2}^2+\|\<t-r\>\nab Z^a H\|_{L^2}^2 \Big),\\
	&\XX^\phi_{j}(t):=\sum_{|\al|\leq j-1} \|\<t-r\> |D^2 Z^a\phi\|_{L^2}^2.
\end{align*}
The main proposition we will prove is as follows.
\begin{prop}\label{XX-prop}
	Let $N\geq 9$. Suppose that $(u,H,\phi)$ is a solution of \eqref{EL-elas-1} satisfying the assumption $E_{N-2}\les \ep^2$. Then
	\begin{align}\label{L2uH}
		&\XX^{uH}_{N} (t)\les E_N(t),\qquad \XX^{uH}_{N-2} (t)\les E_{N-2}(t),\\
		\label{XX-bound}
		&\XX^\phi_{N}(t)\les E_N(t),\qquad \ \XX^\phi_{N-2} (t)\les E_{N-2}(t).
	\end{align}
\end{prop}

\medskip

To start with, we pay our attention to the $L^2$ weight estimates in \eqref{XX-bound}, whose proof  relies on the following  two useful lemmas.
\begin{lemma}  \label{Xphi}
	There holds for any $j\leq N$
	\begin{equation*}
		\XX^\phi_{j}\les E^\phi_{j}+\sum_{|a|\leq j-1}\|\<t+r\>(\d_t^2-\De)Z^{a}\phi\|_{L^2}^2.
	\end{equation*}
\end{lemma}
\begin{proof}
	The readers can refer to Lemma 2.3 and Lemma 3.1 in \cite{KS96} for the proof.
\end{proof}

From this lemma, we shall control the nonlinear terms of $Z^a\phi$-equation in \eqref{EL-VF}. Denote the total nonlinear terms as
\begin{equation}\label{tha}
	\begin{aligned}
		\tilde h_a:&= -\sum_{b+c=a}C_a^b \d_tZ^b u\cdot\nab Z^c\phi-\sum_{b+c=a}C_a^bZ^b u\cdot\nab\d_tZ^c\phi\\
		&\quad -\sum_{b+c+e=a}C_a^{b,c}Z^bu\cdot\nab(Z^cu\cdot\nab Z^e\phi)+\mathcal R_{2;a}.
	\end{aligned}
\end{equation}
\begin{lemma}  \label{4.3}
	For all multi-index $a$, there holds
	\begin{equation}     \label{(t+r)ha}
\begin{aligned}
		\|(t+r)\tilde h_a\|_{L^2}^2\les& (E_{|a|+1}+\XX^\phi_{|a|+1})E_{[|a|/2]+3}(1+E_{[|a|/2]+3}) \\&+E_{|a|+1}(1+E_{[|a|/2]+3})\XX^\phi_{[|a|/2]+3}.
\end{aligned}
	\end{equation}
\end{lemma}
\begin{proof}
	From the $\tilde h_a$ in \eqref{tha}, we have
	\begin{align*}
		\|(t+r)\tilde h_a\|_{L^2}^2&\les \sum_{b+c=a}\|(t+r)\d_t Z^b u\cdot\nab Z^c\phi\|_{L^2}^2+\sum_{b+c=a}C_a^b\|(t+r)Z^b u\cdot\nab\d_tZ^c\phi\|_{L^2}^2\\
		&\quad +\sum_{b+c+e=a}\|(t+r)Z^bu\cdot\nab(Z^cu\cdot\nab Z^e\phi)\|_{L^2}^2+\|(t+r)\RR_{2;a}\|_{L^2}^2\\
		:&=I_1+I_2+I_3+I_4.
	\end{align*}

First, we consider the estimate to  $I_1$
\begin{align}   \label{I1}
	I_1\les \|(t+r)\d_t Z^{|a|}u \nab Z^{[|a|/2]}\phi\|_{L^2}^2+\|(t+r)\d_t Z^{[|a|/2]}u \nab Z^{|a|}\phi\|_{L^2}^2.
\end{align}
In the integral region $r\geq 2\<t\>/3$, by \eqref{AwayCone}, we estimate the right hand side of \eqref{I1} as
\begin{equation}\label{I1-r>}
	\begin{aligned}
		&\|\d_t Z^{|a|}u\|_{L^2}^2\|r\nab Z^{[|a|/2]}\phi\|_{\lf}^2+\|r\d_t Z^{[|a|/2]}u\|_{\lf}\|\nab Z^{|a|}\phi\|_{L^2}^2\\
		&\les E_{|a|+1}E_{[|a|/2]+2}+E_{[|a|/2]+3}E_{|a|}.
	\end{aligned}
\end{equation}
 In the region $r<2\<t\>/3$, using \eqref{decay-1} and \eqref{decay-2}, the right hand side of \eqref{I1} is bounded by
\begin{align*}
	&\|\d_t Z^{|a|}u\|_{L^2}^2 \|\<t\>\nab Z^{[|a|/2]}\phi\|_{\lf(r<2\<t\>/3)}^2+\|\d_t Z^{[|a|/2]}u\|_{L^3}\|\<t\>\nab Z^{|a|}\phi \|_{L^6(r<2\<t\>/3)}\\
	&\les E_{|a|+1} \big(E_{[|a|/2]}+\XX_{[|a|/2]+2}\big)+E_{[|a|/2]+2}
\big(E_{|a|}+\XX_{|a|+1}\big).
\end{align*}

\medskip
Then we consider  $I_2$. For the case $r\geq 2\<t\>/3$, along the exact same lines as in the proof of \eqref{I1-r>}, we derive
\begin{align*}
	&\sum_{b+c=a}C_a^b\|(t+r)Z^b u\cdot\nab\d_tZ^c\phi\|_{L^2(r\geq 2\<t\>/3)}^2\\
	&\les \|Z^{|a|}u\|_{L^2}^2\|r\nab \d_t Z^{[|a|/2]}\phi\|_{\lf}^2+\|r Z^{[|a|/2]}u\|_{\lf}\|\nab \d_t Z^{|a|}\phi\|_{L^2}^2\\
	&\les E_{|a|}E_{[|a|/2]+3}+E_{[|a|/2]+2}E_{|a|+1}.
\end{align*}
For the case $r< 2\<t\>/3$, from \eqref{decay-1} and the Sobolev embedding theorem, we arrive at
\begin{align*}
	&\sum_{b+c=a}C_a^b\|(t+r)Z^b u\cdot\nab\d_tZ^c\phi\|_{L^2(r< 2\<t\>/3)}^2\\
	&\les \| Z^{|a|}u\|_{L^2}^2 \|\<t\>\nab\d_t Z^{[|a|/2]}\phi\|_{\lf(r<2\<t\>/3)}^2+\| Z^{[|a|/2]}u\|_{\lf}^2\|\<t-r\>\nab \d_t Z^{|a|}\phi \|_{L^2(r<2\<t\>/3)}^2\\
	&\les E_{|a|} \big(E_{[|a|/2]+1}+\XX_{[|a|/2]+3}\big)+E_{[|a|/2]+2}\XX_{|a|+1}.
\end{align*}

In a similar way, for $I_3$, we can also obtain that
\begin{align*}
	I_3\les E_{|a|+1}E_{[|a|/2]+3}\big(E_{[|a|/2]+3}+\XX^\phi_{[|a|/2]+3}\big)
+\big(\XX^\phi_{|a|+1}+E_{|a|+1}\big)E_{[|a|/2]+3}^2.
\end{align*}

Finally, we bound  $I_4$ as
	\begin{equation}\label{I4-est}
		\begin{aligned}
			I_4=&\|(t+r)\RR_{2;a}\|_{L^2}^2\\
			 \les &\ \sum_{b+c+e=a}\|(t+r)Z^b\tan\phi_2(D_t Z^c\phi_1 D_t Z^e\phi_2-\nab Z^c\phi_1\nab Z^e\phi_2)\|_{L^2}^2\\
			&+\sum_{b+c+e=a}\|(t+r)Z^b\sin 2\phi_2(-D_t Z^c\phi_1 D_t Z^e\phi_1+\nab Z^c\phi_1 \nab Z^e\phi_1)\|_{L^2}^2.
		\end{aligned}
	\end{equation}
 In the region $r<2\<t\>/3$, we obtain from \eqref{decay-1} and \eqref{decay-2} that
\begin{align*}
	&\|(t+r)\RR_{2;a}\|_{L^2(r<2\<t\>/3)}^2\\
	&\les \|\<t\>\mathcal R_{2;a}\|_{L^2(r<2\<t\>/3)}^2\\
	&\les \|Z^{|a|/2}(\tan\phi,\sin 2\phi)\|_{\lf}^2 \|(D_t Z^{|a|}\phi,\nab Z^{|a|}\phi)\|_{L^2}^2\|\<t\>(D_tZ^{|a|/2}\phi,\nab Z^{|a|/2}\phi)\|_{\lf(r<2\<t\>/3)}^2\\
	&\quad + \|Z^{|a|}(\tan\phi,\sin 2\phi)\|_{L^6}^2 \|(D_tZ^{|a|/2}\phi,\nab Z^{|a|/2}\phi)\|_{L^6}^2\|\<t\>(D_t Z^{ |a|/2}\phi,\nab Z^{ |a|/2}\phi)\|_{L^6(r<2\<t\>/3)}^2\\
	&\les E^{\phi }_{[|a|/2]+1}E^\phi_{|a|}\big(\XX^\phi_{[|a|/2]+2}+E^\phi_{[|a|/2]+2}\big).
\end{align*}
In the region $r\geq 2\<t\>/3$, by \eqref{AwayCone}, we infer
\begin{align*}
	&\|(t+r)\mathcal R_{2;a}\|_{L^2(r\geq 2\<t\>/3)}^2\\
	&\les \| r\mathcal R_{2;a}\|_{L^2(r\geq 2\<t\>/3)}^2\\
	&\les \|Z^{|a|/2}(\tan\phi,\sin 2\phi)\|_{\lf}^2 \|(D_tZ^{|a|}\phi,\nab Z^{|a|}\phi)\|_{L^2}^2\|r(D_tZ^{ |a|/2}\phi,\nab Z^{|a|/2}\phi)\|_{\lf(r\geq 2\<t\>/3)}^2\\
	&\quad + \|Z^{ |a|}(\tan\phi,\sin 2\phi)\|_{L^6}^2 \|(D_tZ^{|a|/2}\phi,\nab Z^{|a|/2}\phi)\|_{L^3}^2\|r (D_tZ^{ |a|/2}\phi,\nab Z^{ |a|/2}\phi)\|_{L^\infty(r\geq 2\<t\>/3)}^2\\
	&\les E^{\phi }_{[|a|/2]+1}E^\phi_{|a|}E^\phi_{[|a|/2]+2}.
\end{align*}

Combining the estimates of $I_1,\cdots,I_4$ gives the bound \eqref{(t+r)ha}. This completes the proof of Lemma \ref{4.3}.
\end{proof}

\medskip
With the above two lemmas in hand, we proceed to prove the bound \eqref{XX-bound}.
\begin{proof}[Proof of the estimates in \eqref{XX-bound}]
	Noticing the third equation of \eqref{EL-VF}, applying Lemma \ref{Xphi} with $j=N-2$ and Lemma \ref{4.3}, we deduce
	\begin{align*}
		\XX^\phi_{N-2}&\les E_{N-2}+\sum_{|a|\leq N-3}\|(t+r)\tilde h_a\|_{L^2}^2\\
		&\les E_{N-2}+(E_{N-2}+\XX^\phi_{N-2})E_{N-2}(1+E_{N-2})
		+E_{N-2}(1+E_{N-2})\XX^\phi_{N-2}\\
		&\les E_{N-2}(1+E^2_{N-2})+E_{N-2}(1+E_{N-2})\XX^\phi_{N-2},
	\end{align*}
which together with $E_{N-2}\les \ep^2$ implies
\begin{equation}  \label{XX_ka-2re}
	\XX^\phi_{N-2}\les E_{N-2}.
\end{equation}
	Then Lemma \ref{Xphi} with $j=N$, Lemma \ref{4.3}, and the estimate \eqref{XX_ka-2re} give
	\begin{align*}
		\XX^\phi_{N}&\les E_{N}+\sum_{|a|\leq N-1}\|(t+r)\tilde h_a\|_{L^2}^2\\
		&\les E_{N}+(E_{N}+\XX^\phi_{N})E_{N-2}(1+E_{N-2}^3) +E_{N}(1+E_{N-2}^3)\XX^\phi_{N-2}\\
		&\les E_N+\ep^2(E_N+\XX^\phi_N).
	\end{align*}
	The above inequality leads to
	\begin{align*}
		\XX^\phi_{N}\les E_N.
	\end{align*}
Thus the desired bounds in \eqref{XX-bound} are obtained.
\end{proof}

\medskip
Next, we turn to the proof of estimates in \eqref{L2uH}. We start with some useful lemmas. Though the following three lemmas have been proved in \cite{LeiWang}, we state them again due to the minor difference.
\begin{lemma}  \label{Qk}
	Assume that $Z^a H$ satisfies \eqref{curl-vf}, then there holds
	\begin{align}   \label{H-(t-r)}
		\|(t-r)\nab Z^a H\|_{L^2}^2\les \|Z^a H\|_{L^2}^2+\|(t-r)\nab \cdot Z^a H\|_{L^2}^2+\QQ_a,
	\end{align}
where
\begin{equation*}
	\QQ_a=\sum_{b+c=a}\int (t-r)^2 \d_j Z^a H_{ik}\big(Z^b H_{lk}\d_l Z^c H_{ij}-Z^b H_{lj}\d_lZ^c H_{ik}\big)\ dx .
\end{equation*}
Moreover,
\begin{align}       \label{Qa-est}
	\QQ_a\les (\XX^H_{|a|+1} +E_{|a|})E^{1/2}_{[|a|/2]+3}+(\XX^{H}_{|a|+1})^{1/2} E^{1/2}_{|a|}(\XX^{H}_{[|a|/2]+3})^{1/2}.
\end{align}
\end{lemma}
\begin{proof}
	The formula \eqref{H-(t-r)} is obtained from \eqref{curl-vf} and integration by parts. Precisely, for the case $a=0$,  we have from  \eqref{curl-vf} that
	\begin{align*}
		\|(t-r)\nab H\|_{L^2}^2&=\int (t-r)^2 \d_j  H_{ik}\d_j  H_{ik}\ dx \\
		&=\int (t-r)^2 \d_j  H_{ik}\d_k H_{ij}+(t-r)^2 \d_j  H_{ik}(\d_j  H_{ik}-\d_k  H_{ij})\ dx\\
		&= \int (t-r)^2 \d_j  H_{ik}\d_k  H_{ij}+(t-r)^2 \d_j  H_{ik}( H_{mk}\d_m  H_{ij}- H_{mj} \d_m  H_{ik})\ dx \\
		&=\int (t-r)^2 \d_j  H_{ik}\d_k  H_{ij}\ dx +\QQ_0.
	\end{align*}
Apply  integration by parts, the first term is rewritten as
\begin{align*}
	\int (t-r)^2 \d_j  H_{ik}\d_k  H_{ij}\ dx&=\int 2(t-r)\om_j   H_{ik}\d_k  H_{ij}- (t-r)^2  H_{ik}\d_j\d_k  H_{ij}\ dx\\
	&=\int 2(t-r)\om_j   H_{ik}\d_k  H_{ij}- 2(t-r) \om_k H_{ik}\d_j  H_{ij}\\
	&\quad +(t-r)^2 \d_k H_{ik}\d_j H_{ij}\ dx\\
	&\leq C\|H\|_{L^2}^2+\frac{1}{2}\|(t-r)\nab H\|_{L^2}^2+C\|(t-r)\nab\cdot H\|_{L^2}^2.
\end{align*}
Hence, we derive
\begin{align*}
	\|(t-r)\nab H\|_{L^2}^2\les \|H\|_{L^2}^2+\|(t-r)\nab\cdot H\|_{L^2}^2+\QQ_0.
\end{align*}
This yields the bound \eqref{H-(t-r)} with $a=0$. Always along the same lines, the estimate \eqref{H-(t-r)} for general $a$ can also be obtained.
	
	For the estimate to formula $\QQ_a$, by the Sobolev embedding theorem, \eqref{AwayCone} and \eqref{Decay-phi2}, we have
	\begin{align*}
		\QQ_a&\les \int \<t-r\>^2 |\nab Z^{a}H|\Big(|Z^{[|a|/2]}H\nab Z^{|a|}H|+|Z^{|a|}H\nab Z^{[|a|/2]}H|\Big)\ dx\\
		&\les \|\<t-r\>\nab Z^{a}H\|_{L^2}^2\||Z^{[|a|/2]}H\|_{\lf}\\
		&\quad +\|\<t-r\>\nab Z^{a}H\|_{L^2}\| Z^{|a|}H\|_{L^2}\|\<t-r\>\nab Z^{[|a|/2]}H\|_{\lf}\\
		&\les \XX^H_{|a|+1} E_{[|a|/2]+2}^{1/2}+(\XX_{|a|+1}^{H})^{1/2}E^{1/2}_{|a|}
\big(E^{1/2}_{[|a|/2]+3}+(\XX^{H}_{[|a|/2]+3})^{1/2}\big)\\
		&\les \big(\XX^H_{|a|+1} +E_{|a|}\big)E^{1/2}_{[|a|/2]+3}+(\XX^{H}_{|a|+1})^{1/2} E^{1/2}_{|a|}(\XX^{H}_{[|a|/2]+3})^{1/2}.
	\end{align*}
This completes the proof of Lemma \ref{Qk}.
\end{proof}

\begin{lemma}\label{Lem5.5}
	Assume that $(u,H,\phi)$ is the solution of \eqref{EL-VF}.  We denote
	\begin{align}   \nonumber
		L_k&=\sum_{|a|\leq k}(|Z^a u|+|Z^a H|),\\ \label{Nk}
		N_k&=\sum_{|a|\leq k-1}\big(t|f_a|+t|g_a|+(t+r)|\mathcal N_a|+t|\nab Z^a p|\big).
	\end{align}
Then for all $|a|\leq k-1$,
\begin{align}\label{goodun}
	(t\pm r)|\nab Z^a u\pm \nab \cdot Z^a H\otimes \om |\les L_k+N_k.
\end{align}
\end{lemma}
\begin{proof}
    The formula \eqref{goodun} is obtained from \eqref{curl-vf}, \eqref{EL-VF} and the decomposition \eqref{decomposition}. We can refer to \cite[Lemma 6.3]{LeiWang} for the detail.
\end{proof}

\begin{lemma}\label{NK0}
	Let $N_k$ be the nonlinear term  \eqref{Nk}. Then we have
	\begin{align}     \label{Nk-est}
			\|N_k\|_{L^2}&\les  E^{1/2}_kE^{1/2}_{[k/2]+3}+E_{[k/2]+2}^{1/2}\XX_k^{1/2}+E_{k-1}^{1/2}\XX_{[k/2]+3}^{1/2}.
	\end{align}
\end{lemma}
\begin{proof}
	Applying the operator $\nab\cdot$ to the $Z^a u$-equation in \eqref{EL-VF} leads to
	\begin{align*}
		\De Z^a p=\nab \cdot f_a,
	\end{align*}
which gives
\begin{align*}
	\|\nab Z^a p\|_{L^2}\les \|\De^{-1}\nab (\nab\cdot f_a)\|_{L^2}\les \|f_a\|_{L^2}.
\end{align*}
Thus we have
\begin{align*}
	\|N_k\|_{L^2}\les \sum_{|a|\leq k-1}\Big(t\|f_a\|_{L^2}+t\|g_a\|_{L^2}+\|(t+r)\mathcal N_a\|_{L^2}\Big).
\end{align*}

\emph{Case 1): the region $r<2\<t\>/3$.}

According to the definitions of $f_a$, $g_a$ and $\mathcal N_a$, (see \eqref{fa}, \eqref{ga}, \eqref{Na}), we have
\begin{align*}
	&\sum_{|a|\leq k-1} \Big( t\|f_a\|_{L^2(r<2\<t\>/3)}+t\|g_a\|_{L^2(r<2\<t\>/3)}+(t+r)\|\mathcal N_a\|_{L^2(r<2\<t\>/3)}\Big)\\
	&\les \sum_{b+c=a;|a|\leq k-1}\<t\>\big\|(|Z^b u|+|Z^b H|+|\nab Z^b \phi|)(|\nab Z^c u|+|\nab Z^c H|+|\nab^2 Z^c\phi|)\big\|_{L^2(r<2\<t\>/3)}\\
	&\les \big\||Z^{[(k-1)/2]} u|+|Z^{[(k-1)/2]} H|+|\nab Z^{[(k-1)/2]} \phi|\big\|_{\lf}\\
	&\qquad \times\big\|\<t-r\>(|\nab Z^{k-1} u|+|\nab Z^{k-1} H|+|\nab^2 Z^{k-1}\phi|)\big\|_{L^2(r<2\<t\>/3)}\\
	&\quad + \big\||Z^{k-1} u|+|Z^{k-1} H|+|\nab Z^{k-1} \phi|\big\|_{L^2}\\
	&\qquad \times\big\|\<t\>(|\nab Z^{[(k-1)/2]} u|+|\nab Z^{[(k-1)/2]} H|+|\nab^2 Z^{[(k-1)/2]}\phi|)\big\|_{L^\infty(r<2\<t\>/3)}\\
	&\les E_{[(k-1)/2]+2}^{1/2}\XX_k^{1/2}+E_{k-1}^{1/2}
\big(E_{[(k-1)/2]+1}^{1/2}+\XX_{[(k-1)/2]+3}^{1/2}\big).
\end{align*}

\emph{Case 2): $r\geq 2\<t\>/3$.}

In view of the decomposition \eqref{decomposition} of $\nab$, we obtain
\begin{align*}
	&|f_a|+|g_a|+|\mathcal N_a|\\
	&\les \sum_{b+c=a;|a|\leq k-1}\Big(|Z^b u|+|Z^b H|+|\nab Z^b \phi|\Big)\Big(|\d_r Z^c u|+|\d_r Z^c H|+|\d_r\nab Z^c\phi|\Big)\\
	&\quad +\sum_{b+c=a;|a|\leq k-1}\frac{1}{r}\Big(|Z^b u|+|Z^b H|+|\nab Z^b \phi|\Big)\Big(|\Om Z^c u|+|\Om Z^c H|+|\Om\nab Z^c\phi|\Big),
\end{align*}
which together with \eqref{AwayCone} implies
\begin{align*}
	&\sum_{|a|\leq k-1}\Big(t\|f_a\|_{L^2(r\geq 2\<t\>/3)}+t\|g_a\|_{L^2(r\geq 2\<t\>/3)}+(t+r)\|\mathcal N_a\|_{L^2(r\geq 2\<t\>/3)}\Big)\\
	&\les \sum_{b+c=a;|a|\leq k-1}\big\|r(|Z^b u|+|Z^b H|+|\nab Z^b \phi|)(|\d_r Z^c u|+|\d_r Z^c H|+|\d_r\nab Z^c\phi|)\big\|_{L^2(r\geq 2\<t\>/3)}\\
	&\quad +\sum_{b+c=a;|a|\leq k-1}\big\|(|Z^b u|+|Z^b H|+|\nab Z^b \phi|)(|\Om Z^c u|+|\Om Z^c H|+|\Om\nab Z^c\phi|)\big\|_{L^2(r\geq 2\<t\>/3)}\\
	&\les E^{1/2}_k E^{1/2}_{[k/2]+3}.
\end{align*}
This completes the proof of Lemma \ref{NK0}.
\end{proof}

\medskip

Utilize the above three lemmas, we then prove the estimates in \eqref{L2uH}.
\begin{proof}[Proof of the estimates in \eqref{L2uH}]
	Noticing the definition of $\XX_k$, applying $\<t-r\>\les 1+|t-r|$ and \eqref{H-(t-r)}, we have
	\begin{align}       \nonumber
		\XX^{uH}_{k}&=\sum_{|a|\leq k-1} \Big(\|\<t-r\>\nab Z^a u\|_{L^2}^2+\|\<t-r\>\nab Z^a H\|_{L^2}^2\Big)\\\nonumber
		&\les E_{k}+\sum_{|a|\leq k-1} \Big(\||t-r|\nab Z^a u\|_{L^2}^2+\||t-r|\nab Z^a H\|_{L^2}^2\Big)\\  \label{XuH-est}
		&\les E_{k}+\sum_{|a|\leq k-1} \Big(\||t-r|\nab Z^a u\|_{L^2}^2+\||t-r|\nab\cdot Z^a H\|_{L^2}^2\Big)+\sum_{|a|\leq k-1}\QQ_a.
	\end{align}
Thanks to
\begin{align*}
	\nab Z^a u=\frac{1}{2}(\nab Z^a u+\nab \cdot Z^a H\otimes \om)+\frac{1}{2}(\nab Z^a u-\nab \cdot Z^a H\otimes \om)
\end{align*}
and
\begin{align*}
	\nab \cdot Z^a H=\frac{1}{2}(\nab Z^a u+\nab \cdot Z^a H\otimes \om)\om-\frac{1}{2}(\nab Z^a u-\nab \cdot Z^a H\otimes \om)\om,
\end{align*}
we obtain
\begin{align*}
	|t-r|(|\nab Z^a u|+|\nab\cdot Z^a H|)&\les |t+r||\nab Z^a u+\nab \cdot Z^a H\otimes \om|\\
	&\quad +|t-r||\nab Z^a u-\nab \cdot Z^a H\otimes \om|.
\end{align*}
The above estimate together with \eqref{goodun} implies that, for any $|a|\leq k-1$,
\begin{align}   \label{second-term}
	\|(t-r)\nab Z^a u\|_{L^2}^2+\|(t-r)\nab\cdot Z^a H\|_{L^2}^2\les E_k+\|N_k\|_{L^2}^2.
\end{align}
Collecting the estimates \eqref{XuH-est} and \eqref{second-term},  applying \eqref{Nk-est} and \eqref{Qa-est}, we arrive at
\begin{equation}   \label{X-k}
\begin{aligned}
	\XX^{uH}_{k}&\les E_{k}+\|N_k\|_{L^2}^2+\sum_{|a|\leq k-1}\QQ_a\\
	&\les  E_{k}+E_kE_{[k/2]+3}+E_{[k/2]+2}\XX_k+E_{k-1}\XX_{[k/2]+3}\\
	&\quad + \XX^H_{k} E^{1/2}_{[k/2]+3}+E_{k-1}E^{1/2}_{[k/2]+3}+(\XX^{H}_{k})^{1/2} E^{1/2}_{k-1}(\XX^{H}_{[k/2]+3})^{1/2}.
\end{aligned}	
\end{equation}

Hence, for the case $k=N-2\geq 7$, by $E_{N-2}\les \ep^2$ and \eqref{XX-bound}, the bound \eqref{X-k} yields
\begin{align*}
	\XX^{uH}_{N-2}\les E_{N-2}+\ep(\XX^{uH}_{N-2}+\XX^\phi_{N-2})\les E_{N-2}+\ep\XX^{uH}_{N-2}.
\end{align*}
Then we get
\begin{align}   \label{XXka-2}
	\XX^{uH}_{N-2}\les E_{N-2}.
\end{align}
For the case $k=N$, using \eqref{XXka-2}, $E_{N-2}\les \ep^2$ and \eqref{XX-bound}, we deduce that
\begin{align*}
	\XX^{uH}_{N} \les E_N+(E_N+\XX^{uH}_{N}+\XX^\phi_{N})\ep^2+\ep\XX^{uH}_{N}+\ep E_{N-1}\les E_N+\ep \XX^{uH}_{N},
\end{align*}
which gives
\begin{equation*}
	\XX^{uH}_{N}\les E_N.
\end{equation*}
Therefore the desired bounds in \eqref{L2uH} are obtained.
\end{proof}

\bigskip
\section{Higher-order energy estimates}\label{sec-HighEnergy}

This section is devoted to the higher-order energy estimate.
Let $N\geq 9$. For any $|a|\leq N$, we recall the following modified energy functional
\begin{align*}
	\bE_a(t):&=E^{uH}_a(t)+E^{\phi}_{a}(t)+\EE^{\phi}_{a}(t)+\bR^{\phi}_{a;1}(t)+\bR^{\phi}_{a;2}(t)\\
	:&=\frac{1}{2}\int |Z^a u|^2+|Z^a H|^2\ dx
	+\frac{1}{2}\int |D_t Z^a \phi|^2+ |\nab Z^a \phi|^2\ dx\\
	&\quad  +\int \frac{1}{2} |Z^a u\cdot\nab\phi|^2+D_t Z^a\phi (Z^au\cdot\nab\phi)\ dx\\
	&\quad -\frac{1}{2} \int \sin^2 \phi_2( |D_t Z^a \phi_1|^2+|\nab Z^a \phi_1|^2)\ dx\\
	&\quad -\int \frac{1}{2}\sin^2\phi_2|Z^au\cdot \nab \phi_1|^2+\sin^2\phi_2 D_t Z^a\phi_1 Z^a u\cdot\nab\phi_1\ dx ,
\end{align*}
with $|Z^a H|^2=\sum_{i,j}Z^a H_{ij} Z^aH_{ij}$. Denote
\[  \bE_m(t)=\sum_{|a|\leq m}\bE_a(t).         \]
For sake of convenience, we also define the energy functional of $Z^a \phi$ as
\begin{equation*}
	\bE^{\phi}_a(t):=E^{\phi}_{a}(t)+\EE^{\phi}_{a}(t)+\bR^{\phi}_{a;1}(t)+\bR^{\phi}_{a;2}(t).
\end{equation*}

Then the high order energy estimates are as follows:
\begin{prop}     \label{Prop_va}
	Let $N\geq 9$. Assume that $(u,H,\phi)$ is the solution of \eqref{EL-elas-1} satisfying \eqref{MainAss_dini} and $E^{1/2}_{N-2}\les \ep$. Then for any $t\in[0,T]$,
	we have the following properties:
	
	i) equivalence relation: for any $0\leq m\leq N$
	\begin{equation}  \label{eq-relation}
		\bE_m(t)\approx_\ep E_m(t).
	\end{equation}
	
	ii) energy estimate:
	\begin{equation}       \label{Ev}
		\frac{d}{dt}\bE_N(t)\lesssim \<t\>^{-1}\bE_N E^{1/2}_{N-2}(1+E_{N-2}).
	\end{equation}
\end{prop}

In what follows, we are aimed at the proof of Proposition \ref{Prop_va}. To start with, we show the equivalence \eqref{eq-relation}.

\vskip 0.5cm
\begin{proof}[Proof of equivalence \eqref{eq-relation}]
	\
	
	In view of  the assumption $E^{1/2}_{N-2}\les \ep$, we have
	\begin{equation*}
		\|\nab\phi\|_{\lf}+\|\sin^2\phi_2\|_{\lf}\les \ep.
	\end{equation*}
    Then by the expression of $\EE^\phi_{a;2}(t)$, $\bR^{\phi}_{a;1}(t)$ and $\bR^{\phi}_{a;2}(t)$, we obtain
	\begin{equation*}
		\EE^\phi_{a}(t)+\bR^{\phi}_{a;1}(t)+\bR^{\phi}_{a;2}(t)\les \ep E_a(t),
	\end{equation*}
which yields the equivalence \eqref{eq-relation}.
\end{proof}	
	
	\medskip
	
Next we prove the energy estimate \eqref{Ev}. This bound is obtained by the following two energy estimates of $(u,H,\phi)$: for any $|a|\leq N$,
\begin{equation}\label{Eu-2}
	\begin{aligned}
		\frac{d}{dt}E^{uH}_a(t)&\leq  C \<t\>^{-1}E_N E^{1/2}_{N-2}(1+E_{N-2})\\
		&\quad -\int Z^a u\cdot  ( \nab \phi\De Z^a\phi-\nab\phi_1 \De Z^a\phi_1 \sin^2\phi_2)\ dx,
	\end{aligned}
\end{equation}
and
\begin{equation} \label{Ephi}
	\begin{aligned}
		\frac{d}{dt}\bE^\phi_{a}(t)
		&\leq C \<t\>^{-1}E_N E^{1/2}_{N-2}(1+E_{N-2})
		\\
		&\quad +\int  Z^a u\cdot  ( \nab \phi\De Z^a\phi-\nab\phi_1 \De Z^a\phi_1 \sin^2\phi_2)\ dx .
	\end{aligned}
\end{equation}
Thus we shall  prove the bounds \eqref{Eu-2} and \eqref{Ephi} separately in the remainder part.

\vskip 0.5cm
\begin{proof}[Proof of the energy estimate \eqref{Eu-2}]
	\
	
	By the first two equations in (\ref{EL-VF}) we calculate
	\begin{align*}
		&\frac{1}{2}\frac{d}{dt}\int  (|Z^a u|^2+| Z^a H|^2)\ dx
		\\
		&=\int  Z^a u\ (\nab\cdot Z^aH-\nab Z^a p+f_a)
		+ Z^a H_{ij}\ (\d_j Z^a u_i+g_{a,ij})\ dx\\
		&=\int (Z^au\cdot f_a+ Z^aH_{ij} g_{a,ij})\ dx,
	\end{align*}
which combined with \eqref{fa} and \eqref{ga} yields
\begin{align*}
	\frac{d}{dt}E^{uH}_a(t)
	&=\int (Z^a u\cdot f_a+ Z^a H\cdot  g_a)\ dx\\
	& =\sum_{b+c=a}C_a^b \int  Z^a u \big(-Z^b u\cdot \nab Z^c u +\nab\cdot(Z^b H Z^c H^T)\big)\ dx\\
	& \quad 	- \sum_{b+c=a}C_a^b\int  Z^a u\ \d_j  (\nab Z^b\phi\cdot \d_j Z^c\phi)\ dx+\int Z^a u\  \mathcal R_{1;a}\ dx\\
	&\quad + \sum_{b+c=a}C_a^b \int  Z^a H_{ij}\cdot \big(-Z^b u\cdot \nab Z^c H_{ij}+\d_k Z^b u_i Z^cH_{kj}\big)\ dx\\
	:&=I_1+I_2+I_3+I_4.
\end{align*}
We then proceed to deal with the nonlinear terms $I_1-I_4$.

\medskip

\emph{1) Estimates of $I_1$ and $I_4$.}

We use $\div u=0$ and $\d_j H_{jk}=0$ to rearrange the $I_1$ and $I_4$ as
\begin{align} \nonumber
	I_1+I_4&= \sum_{b+c=a;|c|<| a|}C_a^b \int  Z^a u_i \big(-Z^b u\cdot \nab Z^c u_i + \d_j Z^c H_{ik} Z^b H_{jk}\big)\ dx\\\nonumber
	&\quad +\sum_{b+c=a:|c|<|a|}C_a^b \int  Z^a H_{ij}\cdot \big(-Z^b u\cdot \nab Z^c H_{ij}+\d_k Z^c u_i Z^b H_{kj}\big)\ dx\\\nonumber
	&\quad +\int  \d_j \big(Z^a u_i Z^a H_{ik} H_{jk}-\frac{1}{2}(u_j|Z^au|^2+u_j |Z^a H|^2)\big)\ dx\\ \label{I1I4}
	:&= I_{1,1}+I_{1,2}+I_{1,3}.
\end{align}
The last integral $I_{1,3}$ vanishes.
Using \eqref{Decay-phi2} and Proposition \ref{XX-prop}, we control the first two terms by
\begin{align*}
	I_{1,1}+I_{1,2}&\les E_N \big(\|Z^{ [N/2]+1}u\|_{\lf}+\|Z^{ [N/2]+1}H\|_{\lf}\big)\\
	&\les E_N \<t\>^{-1}(E^{1/2}_{[N/2]+3}+\XX^{1/2}_{[N/2]+3})\les  \<t\>^{-1}E_N E^{1/2}_{N-2}.
\end{align*}

\medskip
\emph{2) Estimate of $I_2$.}

For the integral $I_2$, by integration by parts, we have
\begin{align} \nonumber
	I_2&=- \sum_{b+c=a}C_a^b\int  Z^a u\  (\nab \d_j  Z^b\phi\cdot \d_j Z^c\phi+\nab Z^b\phi \De Z^c\phi)\ dx\\ \nonumber
	&=- \frac{1}{2}\sum_{b+c=a}C_a^b\int Z^a u\cdot  \nab( \d_j  Z^b\phi\cdot \d_j Z^c\phi)\ dx\\  \nonumber
	&\quad - \sum_{b+c=a;|c|<|a|}C_a^b\int  Z^a u\cdot   \nab  Z^b\phi\De Z^c\phi\ dx
	- \int  Z^a u\cdot   \nab \phi\De Z^a\phi\ dx\\   \label{I33}
	:&= I_{2,1}+I_{2,2}+I_{2,3}.
\end{align}
The integral $I_{2,3}$ can be cancelled by $J_{1,2}$ in \eqref{J11} later, here we retain it temporarily. Apply integration by parts and $\div Z^a u=0$, the term $I_{2,1}$ vanishes.
For the second integral $I_{2,2}$, we obtain from \eqref{Decay-phi2} and Proposition \ref{XX-prop} that
\begin{align*}
	I_{2,2}&\lesssim E^{1/2}_N \big(\|\nab Z^{[N/2]}\phi\|_{\lf}\|\De Z^{N-1}\phi\|_{L^2}+\|\nab Z^{N}\phi\|_{L^2}\|\De Z^{[N/2]}\phi\|_{\lf}\big)\\
	&\les \<t\>^{-1}E_{N}(E^{1/2}_{[N/2]+3}+\XX^{1/2}_{[N/2]+3})\les \<t\>^{-1}E_{N}E^{1/2}_{N-2}.
\end{align*}

\medskip
\emph{3) Estimate of $I_3$.}

The term  $I_3$ is essentially a high order term. By the expression of $\mathcal R_{1;a}$ in \eqref{R1a}, we deduce
\begin{align}\nonumber
	I_3 &\leq \int  Z^a u \sin^2 \phi_2 \nab \phi_1 \De Z^a \phi_1 \ dx+\int  Z^a u \sin^2 \phi_2 \nab \d_j Z^a \phi_1 \d_j \phi_1 \ dx \\\nonumber
	&\quad +C E^{1/2}_{N}\sum_{|b+c+e|\leq |a|+1;|b|,|c|,|e|\leq |a|}\| Z^b\sin^2\phi \nab Z^c\phi\nab Z^e\phi)\|_{L^2}\\ \label{I41}
	:&=I_{3,1}+I_{3,2}+I_{3,3}.
\end{align}
We retain the integral $I_{3,1}$, which will be cancelled by $J_{4,2}$ in \eqref{J63}. For $I_{3,2}$, it follows from $\div Z^a u=0$, \eqref{Zaphi} and \eqref{Decay-phi2} that
\begin{align*}
	I_{3,2}&= -\int Z^a u\ \d_j Z^a \phi_1 \nab(\sin^2 \phi_2  \d_j \phi_1) \ dx \\
	&\les E_N \big(\|\sin^2\phi_2\|_{\lf}\|\nab^2\phi_1\|_{\lf}+\|\nab\phi\|_{\lf}^2\big)\\
	&\les \<t\>^{-5/3+2\de}E_N E_{N-2}(1+E^{1/2}_{N-2}).
\end{align*}
Similarly, the last term $I_{3,3}$ can be bounded by
\begin{align*}
	I_{3,3}&\les E^{1/2}_{N}\Big(\|Z^{ [N/2]}\phi\|_{\lf}^2\|\nab Z^{ [N/2]}\phi\|_{\lf}\|\nab Z^{ N}\phi\|_{L^2}\\
	&\quad +\|Z^{N }\phi\|_{L^{6}}\|Z^{ [N/2]}\phi\|_{\lf}\|\nab Z^{[N/2]}\phi\|_{L^3}\|\nab Z^{ [N/2]}\phi\|_{\lf}\Big)\\
	&\les  E_NE^{1/2}_{N-2}(1+E_{N-2})\<t\>^{-5/3+2\de}.
\end{align*}

Collecting the above estimates, we obtain the energy estimate \eqref{Eu-2}.
\end{proof}

\begin{proof}[Proof of the energy estimate \eqref{Ephi}]
	\
	
	The proof of estimate \eqref{Ephi} is divided into two steps. The first step is devoted to the energy estimates of $E^\phi_{a;1}$ and $\EE^\phi_{a;2}$, and the second step mainly focuses on the estimates of $\bR^\phi_{a;1}$ and $\bR^\phi_{a;2}$.
	
	\emph{Step 1: Energy estimates of $E^\phi_{a}$ and $\EE^\phi_{a}$:}
	\begin{equation}  \label{step1}
		\begin{aligned}
			\frac{d}{dt}(E^\phi_{a}(t)+\EE^\phi_{a}(t))
			\leq C\<t\>^{-1}E_NE^{1/2}_{N-2}(1+E_{N-2})+\int \De Z^a\phi Z^a u\cdot\nab\phi\ dx.
		\end{aligned}
	\end{equation}
	
	By the $Z^a\phi$-equation in (\ref{EL-VF}) and $\div u=0$, we calculate
	\begin{align*}
		&\frac{1}{2}\frac{d}{dt}\int  |D_t Z^a \phi|^2+|\nab Z^a \phi|^2\ dx\\
		&=\int D_t Z^a\phi D_t^2 Z^a\phi +\nab Z^a\phi \nab D_t Z^a\phi-\nab Z^a\phi \nab u\cdot\nab Z^a\phi\ dx\\
		&=-\int_{\R^3} D_t Z^a\phi D_t (Z^au \cdot \nab\phi)\ dx+\int D_t Z^a\phi\cdot \overline{h}_a\ dx\\
		&\quad -\int \nab Z^a\phi \nab u\cdot\nab Z^a\phi\ dx\\
		:&=J_1+J_2+J_3,
	\end{align*}
where $\overline{h}_a$ is the nonlinear terms in \eqref{ha} except $-D_t(Z^au\cdot\nab\phi)$, given by
\begin{equation}  \label{ha-2}
\begin{aligned}
	\overline{h}_a:&=- \sum_{b+c=a;|c|<|a|}C_a^b Z^b u\cdot \nab \d_t Z^c\phi+\sum_{b+c=a;|b|,|c|<|a|}C_a^b\d_t (Z^b u\cdot \nab Z^c\phi)  \\
	&\quad - \sum_{b+c+e=a;|c|,|e|<|a|}C_a^{b,c}Z^b u\cdot \nab (Z^c u\cdot \nab Z^e\phi)+\RR_{2;a}.
\end{aligned}	
\end{equation}

	\medskip
\emph{1). Estimate of $J_1$.}

By integration by parts, we have from the $Z^a\phi$-equation in \eqref{EL-VF} that
\begin{align}\nonumber
	J_1&= -\frac{d}{dt}\int D_t Z^a\phi (Z^au\cdot\nab\phi)dx+\int D^2_t Z^a\phi\ Z^a u\cdot \nab \phi\ dx\\\nonumber
	&=-\frac{d}{dt}\int D_t Z^a\phi (Z^au\cdot\nab\phi)dx+\int  \De Z^a\phi\ Z^a u\cdot \nab \phi\ dx \\\nonumber
	&\quad -\int D_t(Z^a u\cdot \nab\phi) Z^au\cdot\nab\phi\ dx+\int \overline{h}_a Z^au\cdot\nab\phi\ dx\\  \label{J11}
	:&=J_{1,1}+J_{1,2}+J_{1,3}+J_{1,4}.
\end{align}
The second integral $J_{1,2}$ is cancelled by $I_{2,3}$ in \eqref{I33}. The first and third integrals $J_{1,1}, J_{1,3}$ are rewritten as
\begin{align*}
	J_{1,1}+J_{1,3}=&-\frac{d}{dt}\int D_t Z^a\phi (Z^au\cdot\nab\phi)dx -\frac{1}{2}\frac{d}{dt}\int |Z^a u\cdot\nab\phi|^2\ dx=-\frac{d}{dt}\EE^\phi_{a}(t).
\end{align*}
For the forth integral $J_{1,4}$, we bound $\overline{h}_a$ first.
It follows from \eqref{Decay-phi2} and Proposition \ref{XX-prop} that
\begin{align}   \nonumber
	\|\overline{h}_a\|_{L^2}&\les E^{1/2}_{|a|} \big(\|Z^{[|a|/2]+1}u\|_{\lf}+\|\nab Z^{[|a|/2]+1}\phi\|_{\lf}\big)\\\nonumber
&\quad \times\big(1+\|Z^{[|a|/2]+1}u\|_{\lf}+\|\nab Z^{[|a|/2]+1}\phi\|_{\lf}\big)\\\nonumber
	&\quad +  \|Z^{[|a|/2]}(\tan\phi,\sin 2\phi)\|_{\lf}\ E^{1/2}_{|a|}\ \|D_tZ^{[|a|/2]}\phi+\nab Z^{[|a|/2]}\phi\|_{\lf}\\\nonumber
	&\quad + \|Z^{|a|}(\tan\phi,\sin 2\phi)\|_{L^6}\ \|D_tZ^{[|a|/2]}\phi+\nab Z^{[|a|/2]}\phi\|_{L^6}^2\\\nonumber
	&\les E^{1/2}_{|a|} \<t\>^{-1}E^{1/2}_{[|a|/2]+3}\big(1+E^{1/2}_{[|a|/2]+3}\big)+\<t\>^{-1/3+\de}E^{1/2}_{[|a|/2]+2}E^{1/2}_{|a|} \<t\>^{-1}E^{1/2}_{[|a|/2]+2}\\\nonumber
	&\quad +E^{1/2}_{|a|} \<t\>^{-4/3}E_{[|a|/2]+2}\\\label{ha-est}
	&\les  \<t\>^{-1}E^{1/2}_{|a|}E^{1/2}_{[|a|/2]+3}(1+E^{1/2}_{[|a|/2]+3}).
\end{align}
Then we estimate $J_{1,4}$ by
\begin{align*}
	J_{1,4}&\les \sum_{|a|\leq N}\|\overline{h}_a\|_{L^2}\|Z^N u\|_{L^2}\|\nab\phi\|_{\lf}\\
	&\les \<t\>^{-1}E^{1/2}_NE^{1/2}_{N-2}(1+E^{1/2}_{N-2})E^{1/2}_N \<t\>^{-1}E^{1/2}_{N-2}\\
	&\les \<t\>^{-2}E_N E^{1/2}_{N-2}(1+E_{N-2}).
\end{align*}

\medskip
\emph{2). Estimates of $J_2$ and $J_3$.}

Noticing the estimate \eqref{ha-est}, we bound $J_2$ as
\begin{align*}
	J_2\les \|D_t Z^a\phi\|_{L^2}\|\overline{h}_a\|_{L^2}\les \<t\>^{-1}E_NE^{1/2}_{N-2}(1+E^{1/2}_{N-2})
\end{align*}
From \eqref{Decay-phi2} and Proposition \ref{XX-prop}, it is easy to check that
\begin{align*}
	J_3=-\int \nab Z^a\phi \nab u\cdot \nab Z^a\phi\ dx\les E_N\|\nab u\|_{\lf}\les \<t\>^{-1}E_N E^{1/2}_{N-2}.
\end{align*}

Collecting the above bounds, we obtain the estimate \eqref{step1}.

\emph{Step 2: Energy estimates of $\bR^\phi_{a;1}$ and  $\bR^\phi_{a;2}$:}
\begin{equation}\label{step2}
\begin{aligned}
	\frac{d}{dt}\big(\bR^\phi_{a;1}(t)+\bR^\phi_{a;2}(t)\big)	&\leq C \<t\>^{-4/3+\de}E_N E^{1/2}_{N-2}(1+E^2_{N-2}) \\
	&\quad -\int  \sin^2\phi_2 \De Z^a\phi_1 Z^au\cdot \nab\phi_1 \ dx.
\end{aligned}	
\end{equation}

In a similar way as  \emph{Step 1}, we arrive at
\begin{align*}
	&\frac{d}{dt} \bR^\phi_{a;1}(t)= \frac{1}{2} \frac{d}{dt}\int  (-\sin^2 \phi_2)( |D_t Z^a \phi_1|^2+|\nab Z^a \phi_1|^2)\ dx\\
	&=\int (-\sin^2 \phi_2) \Big(D_t Z^a\phi_1 D_t^2 Z^a\phi_1 +\nab Z^a\phi_1 \nab D_t Z^a\phi_1-\nab Z^a\phi_1 \nab u\cdot\nab Z^a\phi_1\Big)\\
	&\quad  -\sin\phi_2 \cos\phi_2 D_t\phi_2 \big(|D_t Z^a \phi_1|^2+|\nab Z^a \phi_1|^2\big) \ dx.
\end{align*}
It follows from the integration by parts and $Z^a\phi_1$-equation in \eqref{EL-VF} that
\begin{align*}
	\frac{d}{dt}\bR^\phi_{a;1}(t)&=\int \sin^2 \phi_2 D_t Z^a\phi_1\cdot D_t(Z^a u\cdot\nab\phi_1)\ dx-\int \sin^2\phi_2 D_t Z^a\phi_1\ \overline{h}_{a,1}\ dx\\
	&\quad  +\int \nab(\sin^2\phi_2)D_t Z^a\phi_1 \nab Z^a\phi_1+ \sin^2 \phi_2 \nab Z^a\phi_1 \nab u\cdot\nab Z^a\phi_1 \ dx \\
	&\quad -\int  \sin\phi_2 \cos\phi_2 D_t\phi_2 (|D_t Z^a \phi_1|^2+|\nab Z^a \phi_1|^2) \ dx\\
	:&= J_4+J_5+J_6+J_7,
\end{align*}
where the nonlinear term $\overline{h}_{a,1}$ is the first component of the vector $\overline{h}_a:=(\overline{h}_{a,1},\overline{h}_{a,2})$ in  \eqref{ha-2}.

Using \eqref{Zaphi}, Proposition \ref{XX-prop} and \eqref{ha-est}, the term $J_5$ is bounded by
\begin{align*}
	J_5&\les \|\sin\phi_2\|^2_{\lf}\|D_t Z^a\phi_1\|_{L^2}\|\overline{h}_{a,1}\|_{L^2}\\
	&\les \<t\>^{-2/3+2\de}E_{N-2}E^{1/2}_{N}\<t\>^{-1}E^{1/2}_NE^{1/2}_{N-2}(1+E^{1/2}_{N-2})\\
	&\les \<t\>^{-5/3+2\de}E_N E^{1/2}_{N-2}(1+E^{3/2}_{N-2}).
\end{align*}
Similarly, from \eqref{Decay-phi2}, \eqref{Zaphi} and Proposition \eqref{XX-prop}, we deduce
\begin{align*}
	J_6+J_7& \les \|\sin\phi_2\|_{\lf}\||D_t\phi_2|+|\nab\phi_2|\|_{\lf}E_N+\|\sin\phi_2\|_{\lf}^2\|\nab u\|_{\lf}E_N\\
	&\les \<t\>^{-4/3+\de}E_N E_{N-2}.
\end{align*}

\medskip
Next, we consider the term $J_4$.
By integration by parts and \eqref{ha}, we obtain
\begin{align} \nonumber
	J_4&=\frac{d}{dt}\int \sin^2\phi_2 D_t Z^a\phi_1 Z^a u\cdot\nab\phi_1\ dx -\int  \sin^2\phi_2 \De Z^a\phi_1 Z^au\cdot \nab\phi_1 \ dx\\ \nonumber
	&\quad  +\int  \sin^2\phi_2 D_t(Z^a u\cdot\nab\phi_1) Z^au\cdot \nab\phi_1 \ dx -\int  \sin^2\phi_2 \overline{h}_{a1}   Z^au\cdot \nab\phi_1 \ dx\\ \nonumber
	&\quad -\int D_t(  \sin^2\phi_2 )D_t Z^a\phi_1 Z^au\cdot \nab\phi_1 \ dx \\   \label{J63}
	:&= J_{4,1}+J_{4,2}+J_{4,3}+J_{4,4}+J_{4,5}.
\end{align}
The second term $J_{4,2}$ is cancelled by $I_{3,1}$ in \eqref{I41}. By integration by parts, the third term $J_{4,3}$ is bounded as
\begin{align*}
	J_{4,3}&=\int \sin^2\phi_2 \frac{1}{2}D_t |Z^a u\cdot \nab\phi_1|^2\ dx\\
	&= \frac{1}{2}\frac{d}{dt}\int \sin^2\phi_2 |Z^a u\cdot \nab\phi_1|^2\ dx-\frac{1}{2}\int D_t (\sin^2\phi_2) |Z^a u\cdot \nab\phi_1|^2\ dx\\
	&\leq \frac{1}{2}\frac{d}{dt}\int \sin^2\phi_2 |Z^a u\cdot \nab\phi_1|^2\ dx+C\|\sin\phi_2\|_{\lf}\|D_t\phi_2\|_{\lf}E_N \|\nab\phi_1\|_{\lf}^2\\
	&\leq \frac{1}{2}\frac{d}{dt}\int \sin^2\phi_2 |Z^a u\cdot \nab\phi_1|^2\ dx+C\<t\>^{-10/3+\de}E_N E^2_{N-2}.
\end{align*}
The last two terms are estimated by \eqref{ha-est}, \eqref{Zaphi}, \eqref{Decay-phi2} and Proposition \ref{XX-prop}
\begin{align*}
	J_{4,4}+J_{4,5}&\les \|\sin\phi_2\|_{\lf}^2 \<t\>^{-1}E^{1/2}_N  E^{1/2}_{N-2}(1+E^{1/2}_{N-2})E^{1/2}_N \|\nab\phi_1\|_{\lf}\\
	&\quad + \|\sin\phi_2\|_{\lf}\|D_t\phi_2\|_{\lf}E_N \|\nab\phi_1\|_{\lf}\\
	&\les \<t\>^{-8/3+2\de}E_N E_{N-2}^2(1+E^{1/2}_{N-2})+\<t\>^{-7/3+\de}E_N E^{3/2}_{N-2}\\
	&\les \<t\>^{-7/3+\de}E_N E^{1/2}_{N-2}(1+E_{N-2}^2).
\end{align*}

Summing up the above estimates implies \eqref{step2}.
\end{proof}

\bigskip
\section{Lower-order energy estimates}\label{sec-LowEnergy}

This section is devoted to the lower-order energy estimate. We first recall the modified energies
\begin{align*}
	E^{uH}_a(t):& =\frac{1}{2}(\lV  Z^a u\rV_{L^2}^2+\| Z^aH\|_{L^2}^2),\\
	E^\phi_{a}+\EE^\phi_{a}:&=\frac{1}{2}(\|D_t Z^a \phi\|_{L^2}^2+ \|\nab Z^a \phi\|_{L^2}^2)+\int \frac{1}{2} |Z^a u\cdot\nab\phi|^2+D_t Z^a\phi (Z^au\cdot\nab\phi)\ dx.
\end{align*}
The modified energy $\EE_a:=E^{uH}_a+E^\phi_a+\EE^\phi_a$ and $\EE_N:=\sum_{|a|\leq N}\EE_a$.

Then we have the following proposition.
\begin{prop}     \label{LE-prop}
	Let $N\geq 9$. Assume that $(u,H,\phi)$ is the solution of \eqref{EL-elas-1} satisfying \eqref{MainAss_dini} and $E^{1/2}_{N-2}\les \ep$. Then for any $t\in[0,T]$, we have energy estimate:
	\begin{equation}       \label{LE}
		\frac{d}{dt} \EE_{N-2}\lesssim \<t\>^{-4/3+\de}E_{N-2}E_N^{1/2}(1+E_{N-2}),
	\end{equation}
where $\EE_{N-2}$ is equivalent to $E_{N-2}$.
\end{prop}

 Now we begin the proof of Proposition \ref{LE-prop}. In fact,
the estimate \eqref{LE} can be obtained from
\begin{equation}       \label{LEv}
	\frac{d}{dt}E^{uH}_a\leq -\int  (Z^a u\cdot   \nab  \phi)\De Z^a\phi\ dx+ C\<t\>^{-4/3+\de}E_{N-2}E_N^{1/2}(1+E_{N-2}),
\end{equation}
and
\begin{equation}   \label{LEphi}
	\frac{d}{dt}(E^\phi_{a}+\EE^\phi_{a})\leq \int (Z^au\cdot\nab\phi)\De Z^a\phi\ dx+ C\<t\>^{-4/3+\de}E_{N-2}E_N^{1/2}(1+E_{N-2}).
\end{equation}
Hence it suffices to prove the above two estimates.

\begin{proof}[Proof of the energy estimate \eqref{LEv}]
	\

By the first two equations in (\ref{EL-VF}), we calculate
\begin{align*}
	&\frac{1}{2}\frac{d}{dt}\Big(\lV  Z^a u\rV_{L^2}^2+\| Z^aH\|_{L^2}^2\Big)=\int  Z^a u\cdot f_a
	+Z^aH_{ij}\cdot  g_{a,ij} \ dx\\
	& =\sum_{b+c=a}C_a^b \int Z^a u \big(-Z^b u\cdot \nab Z^c u +\nab \cdot (Z^b H Z^cH^T)\big)\\
	&\quad +Z^a H_{ij}\cdot \big(-Z^b u\cdot\nab  Z^c H_{ij}+\d_k Z^b u_i Z^c H_{kj}\big)\ dx\\
	&\quad 	- \sum_{b+c=a}C_a^b\int  Z^a u\ \d_j  (\nab Z^b\phi\cdot \d_j Z^c\phi)\ dx+\int  Z^a u\  \RR_{1;a}\ dx\\
	:&=K_1+K_2+K_3.
\end{align*}
Then we proceed to bound the terms $K_1- K_3$ separately.

\medskip
\emph{1) Estimate of $K_1$.}

Similar to \eqref{I1I4}, using $\div u=0$, $\d_j H_{jk}=0$ and integration by parts, we rewrite the $K_1$ as
\begin{align*}
	K_1&= \sum_{b+c=a;|c|<| a|}C_a^b \int  Z^a u_i \big(-Z^b u\cdot \nab Z^c u_i + \d_j Z^c H_{ik} Z^b H_{jk}\big)\ dx\\
	&\quad +\sum_{b+c=a:|c|<|a|}C_a^b \int   Z^a H_{ij}\cdot \big(-Z^b u\cdot \nab Z^c H_{ij}+\d_k Z^c u_i Z^b H_{kj}\big)\ dx.
\end{align*}
 Thanks to the constriants $\nab \cdot Z^b u=\nab\cdot Z^b H^T=0$, all of the above terms are dealt with in similar method. Here we only consider the following term in detail
\begin{align*}
	k_1:=\sum_{|a|\leq N-2}\sum_{b+c=a;|c|<| a|}|\d_j Z^c H_{ik} Z^b H_{jk}|.
\end{align*}
 In the region $r<2\<t\>/3$, applying \eqref{Decay-phi2} and Proposition \ref{XX-prop}, we infer
\begin{align*}
	\|k_1\|_{L^2(r<2\<t\>/3)}&\les \<t\>^{-2}\|\<t-r\>\nab Z^{N-3} H\|_{L^2}\|\<t\> Z^{N-2} H \|_{\lf} \\
	&\les E^{1/2}_{N-2}\<t\>^{-2}\XX^{1/2}_{N-2} (E^{1/2}_{N}+\XX^{1/2}_{N})\les \<t\>^{-2}E_{N-2}E^{1/2}_N.
\end{align*}
In the region $r\geq 2\<t\>/3$, we utilize \eqref{decomposition} to write $k_1$ as
\begin{align*}
	k_1=\sum_{|a|\leq N-2}\sum_{b+c=a;|c|<| a|}\Big|\Big(\om_j \d_r-\frac{(\om\times \Om)_j}{r} \Big)Z^c H_{ik} Z^b H_{jk}\Big|.
\end{align*}
Then it follows from \eqref{omuH}, \eqref{Decay-phi2} and Proposition \ref{XX-prop} that
\begin{align*}
	\|k_1\|_{L^2(r\geq 2\<t\>/3)}&\les \<t\>^{-3/2}\|\d_r Z^{N-3}H\|_{L^2}\|r^{3/2}\om_j Z^{N-2}H_{jk}\|_{\lf}+\<t\>^{-1}\|\Om Z^{N-3}H\|_{\lf}\|Z^{N-2}H\|_{L^2}\\
	&\les \<t\>^{-3/2}E^{1/2}_{N-2}E^{1/2}_N
+\<t\>^{-2}\big(E^{1/2}_{N}+\XX^{1/2}_N\big)E^{1/2}_{N-2}\les \<t\>^{-3/2}E^{1/2}_{N-2}E^{1/2}_N.
\end{align*}
In a similar way, we can also bound the other three terms in $K_1$. Hence, we obtain
\begin{align*}
	K_1\les \||Z^a u|+|Z^a H|\|_{L^2}\<t\>^{-3/2}E^{1/2}_{N-2}E^{1/2}_N\les \<t\>^{-3/2}E_{N-2}E^{1/2}_N.
\end{align*}

\medskip
\emph{2) Estimate of $K_2$.}

By integration by parts, we rewrite the $K_2$ as
\begin{align}     \nonumber
	K_2&=- \sum_{b+c=a}C_a^b\int  Z^a u\  (\nab \d_j  Z^b\phi\cdot \d_j Z^c\phi+\nab Z^b\phi \De Z^c\phi)\ dx\\\nonumber
	&=- \frac{1}{2}\sum_{b+c=a}C_a^b\int  Z^a u\cdot  \nab( \d_j  Z^b\phi\cdot \d_j Z^c\phi)\ dx
	- \sum_{b+c=a}C_a^b\int  Z^a u\cdot   \nab  Z^b\phi\De Z^c\phi\ dx\\\nonumber
	&=-\int  Z^a u\cdot   \nab  \phi\De Z^a\phi\ dx- \sum_{b+c=a;|c|<|a|}C_a^b\int  Z^a u\cdot   \nab  Z^b\phi\De Z^c\phi\ dx\\\label{K21}
	:&= K_{2,1}+K_{2,2}.
\end{align}
The first integral in the second line above vanishes by $\div Z^a u=0$. The term $K_{2,1}$ is retained, which can be cancelled by $L_{1,2}$ in \eqref{L21} later.
For the integral $K_{2,2}$, in the region $r<\frac{2}{3}\<t\>$, we have from \eqref{Decay-phi2} and Proposition \ref{XX-prop} that
\begin{align*}
	&- \sum_{b+c=a;|c|<|a|}C_a^b\int_{r<\frac{2}{3}\<t\>}  Z^a u\cdot   \nab  Z^b\phi\De Z^c\phi\ dx\\
	&\les \sum_{b+c=a;|c|<|a|} \<t\>^{-2}\|Z^a u\|_{L^2} \|\<t\>\nab Z^b\phi\|_{\lf}\|\<t-r\>\De Z^c\phi\|_{L^2}\\
	&\les \<t\>^{-2}E^{1/2}_{N-2}(E^{1/2}_{N}+\XX^{1/2}_{N})\XX^{1/2}_{N-2}\les \<t\>^{-2}E_{N-2}E^{1/2}_N.
\end{align*}
 In the region $r\geq \frac{2}{3}\<t\>$, combining \eqref{decomposition}, we further write the $K_{2,2}$ as
\begin{align*}
	\tilde K_{2,2}:&= \sum_{b+c=a;|c|<|a|}C_a^b\int_{r\geq 2\<t\>/3} Z^a u\cdot \Big(\om\d_r  - \frac{x}{r^2}\times \Om\Big) Z^b\phi\De Z^c\phi\ dx.
\end{align*}
Then applying \eqref{omuH}, \eqref{Zaphi}, interpolation inequality and Proposition \ref{XX-prop}, we deduce
\begin{align*}
	\tilde K_{2,2}&\les  \<t\>^{-3/2}\|r^{3/2}(\om\cdot Z^{N-2} u)\|_{\lf}  E_{N-2}\\
	&\quad +\<t\>^{-1}E^{1/2}_{N-2}\Big(\|\Om Z^{[(N-2)/2]}\phi\|_{\lf}\|\De Z^{N-3}\phi\|_{L^2}+\|\Om Z^{N-2}\phi\|_{L^6}\|\De Z^{[(N-2)/2]}\phi\|_{L^3}\Big)\\
	&\les \<t\>^{-3/2}E_{N-2}E^{1/2}_N+\<t\>^{-4/3+\de}E_{N-2}
\big(E^{1/2}_{[(N-2)/2]+3}+\XX^{1/2}_{[(N-2)/2]+3}\big)\\
	&\quad +\<t\>^{-1}E^{1/2}_{N-2}E^{1/2}_{N-1}E_{[(N-2)/2]+1}^{1/3}\<t\>^{-1/3}
\big(E^{1/2}_{[(N-2)/2]+3}+\XX^{1/2}_{[(N-2)/2]+3}\big)^{1/3}\\
	&\les \<t\>^{-4/3+\de}E_{N-2}E^{1/2}_N.
\end{align*}

\medskip 	
	
\emph{3) Estimate of $K_3$.}

The term $K_3$ is essentially a high order term. By \eqref{R1a}, \eqref{Zaphi} and \eqref{Decay-phi2}, we have
\begin{align*}
	K_3 &=\int Z^a u\cdot \sum_{b+c+e=a}C_a^{b,c}\d_j (Z^b\sin^2\phi_2 \nab Z^c\phi_1\d_j Z^e\phi_1)\ dx\\
	&\les E^{1/2}_{N-2}\Big(\|Z^{ [N/2]}(\sin^2\phi)\|_{\lf}^2\|\nab Z^{ [N/2]}\phi\|_{\lf}\|\nab Z^{N-1}\phi\|_{L^2}\\
	&\quad +\|Z^{N-2 }\nab (\sin^2\phi)\|_{L^{6}}\|\nab Z^{[N/2]}\phi\|_{L^3}\|\nab Z^{ [N/2]}\phi\|_{\lf}\Big)\\
	&\les E_{N-2}^{2}E^{1/2}_N\<t\>^{-5/3+2\de}.
\end{align*}
This completes the proof of the estimate \eqref{LEv}.
\end{proof}

\begin{proof}[Proof of the energy estimate \eqref{LEphi}]
\

Using the $ Z^a\phi$-equation in (\ref{EL-VF}) and $\div u=0$, we calculate
\begin{align*}
	&\frac{1}{2}\frac{d}{dt}(\|D_t Z^a\phi\|_{L^2}^2+\|\nab Z^a\phi\|_{L^2}^2)
	= \int D_t Z^a\phi D_t^2 Z^a\phi+\nab Z^a\phi \d_t \nab Z^a\phi\ dx \\
	&= \int D_t Z^a\phi\cdot h_a-\nab Z^a\phi \nab u\cdot \nab Z^a\phi  \ dx\\
	:&=L_1+L_2+L_3,
\end{align*}
where
\begin{align*}
	L_1&=-\int D_t Z^a\phi\cdot D_t( Z^a u\cdot\nab \phi)\ dx,\\
	L_2&=-\int D_t Z^a\phi\cdot \Big(\sum_{b+c=a;|c|<|a|}C_a^bZ^b u\cdot\nab\d_tZ^c\phi+\sum_{b+c=a;|b|,|c|<|a|}C_a^b\d_t( Z^b u\cdot\nab Z^c\phi)\Big)\\
	&\quad +\nab Z^a\phi \nab u\cdot \nab Z^a\phi\ dx,\\
	L_3&=- \int D_t Z^a\phi\cdot \Big(\sum_{b+c+e=a;|c|,|e|<|a|}C_a^{b,c}Z^b u\cdot\nab(Z^c u\cdot\nab Z^e\phi)-\RR_{2;a}\Big)\ dx.
\end{align*}
It remains to bound the nonlinear terms $L_1-L_3$.

\medskip
\emph{1) Estimate of $L_1$.}

By integration by parts, \eqref{EL-VF} and \eqref{ha}, we rewrite $L_1$ as
\begin{align}\nonumber
	L_1&=-\frac{d}{dt}\int D_t Z^a\phi (Z^a u\cdot\nab\phi)\ dx+ \int D_t^2 Z^a\phi (Z^au\cdot\nab\phi)\ dx\\\nonumber
	&=\Big(-\frac{d}{dt}\int D_t Z^a\phi (Z^a u\cdot\nab\phi)\ dx-\frac{1}{2}\frac{d}{dt}\int |Z^au\cdot\nab\phi|^2 \ dx\Big)\\\nonumber
	&\quad + \int \De Z^a\phi (Z^au\cdot\nab\phi)\ dx+\int \overline{h}_a (Z^a u\cdot\nab\phi)\ dx\\     \label{L21}
	:&= L_{1,1}+L_{1,2}+L_{1,3},
\end{align}
where $\overline{h}_a$ is given in \eqref{ha-2}.

The first term $L_{1,1}$ is actually the term $-\frac{d}{dt}\EE^\phi_{a}$. The second term $L_{1,2}$ is cancelled by $K_{2,1}$ in \eqref{K21} above. For the last integral $L_{1,3}$, it follows from \eqref{ha-est}  that
\begin{align*}
	\sum_{|a|\leq N-2}\|\overline{h}_a\|_{L^2}\les \<t\>^{-1}E_{N-2}(1+E^{1/2}_{N-2}),
\end{align*}
which combined with \eqref{Decay-phi2} and Proposition \ref{XX-prop} yields
\begin{align*}
	L_{1,3}&\les \|\overline{h}_a\|_{L^2}\|Z^a u\|_{L^2}\|\nab\phi\|_{\lf}\\
	&\les \<t\>^{-1}E_{N-2}(1+E^{1/2}_{N-2})E^{1/2}_{N-2}\<t\>^{-1}(E^{1/2}_{N-2}+\XX^{1/2}_{N-2})\\
	&\les \<t\>^{-2}E^2_{N-2}(1+E^{1/2}_{N-2}).
\end{align*}

\medskip
\emph{2) Estimate of $L_2$.}

It suffices to bound the integrand
\begin{align*}   
	l_2:= \sum_{|a|\leq N-2}\Big(\sum_{b+c=a;|c|<|a|}C_a^b |Z^b u\cdot\nab\d_tZ^c\phi |+\sum_{b+c=a;|b|<|a|}C_a^b|\nab_{t,x} Z^b u\cdot\nab Z^c\phi|\Big).
\end{align*}
For the case $r<\frac{2}{3}\<t\>$, we have from \eqref{Decay-phi2} and Proposition \ref{XX-prop} that
\begin{align}   \nonumber
	\|l_2\|_{L^2(r<2\<t\>/3)}&\les \<t\>^{-2}\|\<t\>Z^{N-2}u\|_{\lf}\|\<t-r\>\nab \d_t Z^{N-3}\phi\|_{L^2}\\\nonumber
	&\quad +\<t\>^{-2}\|\<t-r\>\d_tZ^{N-3}u\|_{L^2}\|\<t\>\nab Z^{N-2}\phi\|_{\lf}\\\nonumber
	&\les \<t\>^{-2}(E^{1/2}_N+\XX^{1/2}_N)\XX^{1/2}_{N-2}\\ \label{l2-r<}
	&\les \<t\>^{-2}E^{1/2}_NE^{1/2}_{N-2}.
\end{align}
For the case $r\geq \frac{2}{3}\<t\>$, we use \eqref{decomposition} to write $l_2$ as
\begin{align*}
	l_2&=\sum_{b+c=a;|c|<|a|\leq N-2}C_a^b |Z^b u\cdot(\om\d_r-\frac{x}{r^2}\times \Om)\d_tZ^c\phi |\\
	&\quad +\sum_{b+c=a;|b|,|c|<|a|\leq N-2}C_a^b|\d_t Z^b u\cdot(\om\d_r-\frac{x}{r^2}\times \Om)Z^c\phi|\\
	&\les \big(|\om\cdot Z^{N-2} u|| \d_r\d_tZ^{N-3}\phi |+|\om\cdot \d_t Z^{N-3} u||\d_rZ^{N-2}\phi|\big)\\
	&\quad +\big(\frac{1}{r}|Z^{N-2} u||\Om\d_tZ^{N-3}\phi |+\frac{1}{r}|\d_t Z^{N-3} u||\Om Z^{N-3}\phi|\big)\\
	:&=l_{2,1}+l_{2,2}.
\end{align*}
Then from \eqref{omuH}, the $l_{2,1}$ is bounded by
\begin{align} \nonumber
	\|l_{2,1}\|_{L^2(r\geq 2\<t\>/3)}&\les \<t\>^{-3/2}\|r^{3/2}\om \cdot Z^{N-2}u\|_{\lf}\|\d_r\d_t Z^{N-3}\phi\|_{L^2}\\\nonumber
	&\quad +\<t\>^{-3/2}\|r^{3/2}\om \cdot \d_t Z^{N-3}u\|_{\lf}\|\d_r Z^{N-2}\phi\|_{L^2}\\  \label{l21}
	&\les \<t\>^{-3/2} E^{1/2}_N E^{1/2}_{N-2}.
\end{align}
Using \eqref{Decay-phi2}, the Sobolev embedding theorem and Proposition \ref{XX-prop}, the term $l_{2,2}$ is estimated as
\begin{align}\nonumber
	\|l_{2,2}\|_{L^2(r\geq 2\<t\>/3)}&\les \<t\>^{-1}\big(\|Z^{N-2}u\|_{L^2}\|\Om\d_t Z^{N-3}\phi\|_{\lf}+\|\d_t Z^{N-3}u\|_{L^3}\|\Om Z^{N-2}\phi\|_{L^6}\big)\\\nonumber
	&\les \<t\>^{-1}\big(E^{1/2}_{N-2} \<t\>^{-1}(E^{1/2}_N+\XX^{1/2}_N)+E^{1/3}_{N-2} \<t\>^{-1/3}(E^{1/6}_N+\XX^{1/6}_N)E^{1/2}_{N-1}   \big)\\ \label{l22}
	&\les \<t\>^{-4/3}E^{1/2}_{N-2}E^{1/2}_N.
\end{align}
Hence, collecting the bounds \eqref{l2-r<}, \eqref{l21} and \eqref{l22}, we obtain
\begin{align*}
	L_2\les \|D_t Z^a\phi\|_{L^2}\|l_2\|_{L^2}\les \<t\>^{-4/3}E_{N-2}E^{1/2}_N.
\end{align*}

\medskip
\emph{3) Estimate of $L_3$.}

From the estimate for cubic term and $\RR_{2;a}$ in \eqref{ha-est}, under the condition $|a|\leq N-2$ and $N\geq 9$, we have
\begin{align*}
	\sum_{|a|\leq N-2}\Big\|\sum_{b+c+e=a;|c|,|e|<|a|}C_a^{b,c}Z^b u\cdot\nab(Z^c u\cdot\nab Z^e\phi)-\RR_{2;a}\Big\|_{L^2}\les \<t\>^{-4/3+\de}E^{3/2}_{N-2},
\end{align*}
which yields
\begin{align*}
	L_3\les \|D_t Z^{N-2}\phi\|_{L^2}\<t\>^{-4/3+\de}E^{3/2}_{N-2}\les \<t\>^{-4/3+\de}E^2_{N-2}.
\end{align*}
Then we complete the proof of estimate \eqref{LEphi}.
\end{proof}

\bigskip
\section{Proofs of the main theorems}\label{sec-2.3}
We first show the proof of  Theorem \ref{Ori_thm} from the Proposition \ref{Prop_va} and \ref{LE-prop}. Once this is done, we can obtain Theorem \ref{Ori_thm0}  by applying Theorem \ref{Ori_thm}.

\begin{proof}[Proof of Theorem \ref{Ori_thm}]
	It follows from Proposition \ref{Prop_va} and \ref{LE-prop} that
	\begin{equation}       \label{Ev-re}
		\frac{d}{dt}\bE_N (t)\leq C_1 \<t\>^{-1}\bE_N  \EE_{N-2}^{1/2}(1+ \EE_{N-2})
	\end{equation}
	and
	\begin{equation} \label{Ephi-re}
		\begin{aligned}
			\frac{d}{dt}\EE_{N-2}(t)\leq C_2 \<t\>^{-4/3+\de} \EE_{N-2}\bE_N^{1/2}(1+ \EE_{N-2}).
		\end{aligned}
	\end{equation}
	
	In view of \eqref{MainAss_dini}, the initial data satisfies
	\begin{equation*}
		\bE^{1/2}_N(0)\leq 2M,\quad \EE^{1/2}_{N-2}(0)\leq 2\ep\leq 2\ep_0.
	\end{equation*}
	By continuity, there exists a positive $T<\infty$ such that following the bounds of $\bE^{1/2}_N(t)$ and $\EE^{1/2}_{N-2}(t)$ are ture for $t\in[0,T]$,
	\begin{align}     \label{bound}
		\bE^{1/2}_N(t)\leq C_0 M\<t\>^\de,\quad  \EE^{1/2}_{N-2}(t)\leq 2\ep e^{C_0^2 M}\leq C_0^{-1}\de.
	\end{align}
	
	Next, we prove that the bounds in \eqref{bound} are true for all $t\in[0,\infty)$. Precisely, we assume that $T_{max}\in [T,\infty)$ is the largest time such that the bounds in \eqref{bound} hold on $[0,T_{\max}]$. Combining the second bound in \eqref{bound}, energy estimates \eqref{Ev-re} and Gronwall's inequality, we have
	\begin{align*}
		\bE_N\leq \bE_N(0) \exp\Big\{\int_0^t 2C_1C_0^{-1}\de\<s\>^{-1}ds\Big\}\leq 4M^2 \<t\>^{2C_1C_0^{-1}\de}.
	\end{align*}
	In a similar way, by the first bound in \eqref{bound} and \eqref{Ephi-re}, we deduce
	\begin{align*}
		\EE_{N-2}\leq \EE_{N-2}(0)\exp\Big\{\int_0^t 2C_2 C_0 M\<s\>^{-4/3+2\de}ds\Big\}\leq 4\ep^2 e^{8C_2 C_0 M}.
	\end{align*}
	
	Therefore, taking $C_0$ to be sufficiently large to verify
	\begin{align*}
		C_0\geq \max\{4, 4C_1,8C_2\}
	\end{align*}
	and letting $\ep_0>0$ to be sufficiently small to fulfill
	\begin{equation*}
		\ep_0\leq \frac{1}{2}C_0^{-1}\de e^{-C_0^2 M},
	\end{equation*}
	we derive the energy bounds
	\begin{align*}
		\bE^{1/2}_N(t)\leq 2M \<t\>^{\de/2}\leq \frac{C_0}{2}M\<t\>^\de,\quad \EE^{1/2}_{N-2}\leq 2\ep e^{4C_2 C_0M}\leq 2\ep e^{ C^2_0M/2}.
	\end{align*}
	The above inequalities show that \eqref{bound} can still be true for $t\in [0,T_{max}+\ep']$ for
	some $\ep'>0$. This contracts to the assumption on $T_{\max}$. Hence we prove the a priori bounds \eqref{bound} on $[0,\infty)$. By the equivalence relation \eqref{eq-relation}, we then obtain the bounds \eqref{energybd}.
	This completes the proof of Theorem \ref{Ori_thm}.
\end{proof}

\begin{proof}[Proof of Theorem \ref{Ori_thm0}]
	\emph{Step 1. We prove that}
	\begin{align}     \label{ini-phi}
		\|(\psi_{0},\ \psi_{1})\|_{H^N_\Lambda}\les C,\quad \|(\psi_{0},\ \psi_{1})\|_{H^{N-2}_\Lambda}\les \ep.
	\end{align}

	Since $(\phi_1,\phi_2)\big|_{t=0}=(\psi_{0,1},\psi_{0,2})$, then from the initial data \eqref{Main-ini0}, we infer that
	\begin{align*}
		&\int |\nab\psi_{0,1}|^2\cos^2\phi_2+|\nab \psi_{0,2}|^2\ dx+\big\||\nab\psi_{0,1}|^2\cos^2\phi_2+|\nab\psi_{0,2}|^2\big\|_{\lf}\\
		&\les \|\nab d_0\|_{L^2\cap \lf}^2
		\les \|\nab d_0\|_{H^2}^2\les \ep^2.
	\end{align*}
	This implies
	\begin{align*}
		\|\nab\psi_{0,1}\|_{L^2\cap \lf}+\|\nab\psi_{0,2}\|_{L^2\cap \lf}\les \ep.
	\end{align*}
	Then by induction, we have
	\begin{align*}
		\|\nab d_0\|_{H^k}&\gtrsim \|\nab\psi_{0}\|_{H^k}-\sum_{2\leq r\leq k+1}\sum_{l_1+\cdots l_r\leq k-r}\|\nab^{l_1+1}\psi_{0}\cdots \nab^{l_r+1} \psi_{0}\|_{L^2}\\
		&\gtrsim \|\nab\psi_{0}\|_{H^k}-\ep \|\nab\psi_{0}\|_{H^{k-1}},
	\end{align*}
	which yields
	\begin{equation*}
		\|\nab \psi_{0}\|_{H^{N}}\les C,\quad \|\nab \psi_{0}\|_{H^{N-2}}\les \ep.
	\end{equation*}
	Along the same lines, we can also calculate
	\begin{align*}
		\sum_{|a|\leq j}\|\nab \La^a d_0\|_{L^2}&\gtrsim \sum_{|a|\leq j}\|\nab \La^a \psi_{0}\|_{L^2}-\sum_{2\leq r\leq j+1}\sum_{l_1+\cdots l_r\leq j+1-r}\|\nab \La^{l_1}\psi_{0} \La^{l_2+1}\psi_{0}\cdots \La^{l_r+1} \psi_{0}\|_{L^2}\\
		&\gtrsim \sum_{|a|\leq j}\|\nab \La^a \psi_{0}\|_{L^2}-\|\nab\psi_{0}\|_{H^1}\sum_{|a|\leq j}\|\nab \La^a \psi_{0}\|_{L^2}\\
		&\quad -\sum_{|a|\leq j-1}\|\nab \La^a \psi_{0}\|_{L^2}^2\big(1+\|\nab \La^a \psi_{0}\|_{L^2}^{j-1}\big).
	\end{align*}
	By induction, this gives
	\begin{align*}
		\sum_{|a|\leq N-2}\|\nab \La^a \psi_{0}\|_{L^2}\les \ep,\quad \sum_{|a|\leq N}\|\nab \La^a \psi_{0}\|_{L^2}\les C.
	\end{align*}
	Hence the bound \eqref{ini-phi} for $\psi_0$ is obtained. The bound \eqref{ini-phi} for $\psi_1$ can be proved similarly.

	\medskip
	\emph{Step 2. We prove the global well-posedness of \eqref{ori_sys}.}
	By \emph{Step 1}, we show that the initial data $(u_0,H_0,\psi_{0},\psi_{1})$ satisfies \eqref{MainAss_dini}. Then from Theorem \ref{Ori_thm}, there existes a unique global solution $(u,H,\phi)$ satisfying \eqref{energybd}. By the expressions
	\begin{equation*}
		F=H+I,\quad d=(\cos\phi_1\cos\phi_2,\sin\phi_1\cos\phi_2,\sin\phi_2),
	\end{equation*}
	we obtain the unique solution $(u,F,d)$ to \eqref{ori_sys}. Moreover,
	\begin{align*}
		\sum_{|a|\leq N-2}\|\nab Z^a d\|_{L^2}&\les \sum_{|a|\leq N-2}\|\nab Z^a \phi\|_{L^2}+\sum_{2\leq r\leq N-1}\sum_{l_1+\cdots l_r\leq N-1-r}\|\nab Z^{l_1}\phi Z^{l_2+1}\phi\cdots Z^{l_r+1} \phi\|_{L^2}\\
		&\les \sum_{|a|\leq N-2}\|\nab Z^a \phi\|_{L^2}\les \ep.
	\end{align*}
	Similarly, we also have
	\begin{align*}
		\sum_{|a|\leq N}\|\nab Z^a d\|_{L^2}&\les \sum_{|a|\leq N}\|\nab Z^a \phi\|_{L^2}+\sum_{2\leq r\leq N+1}\sum_{l_1+\cdots l_r\leq N+1-r}\|\nab Z^{l_1}\phi Z^{l_2+1}\phi\cdots Z^{l_r+1} \phi\|_{L^2}\\
		&\les \sum_{|a|\leq N}\|\nab Z^a \phi\|_{L^2}\les \<t\>^\de.
	\end{align*}
	Therefore the energy bounds in \eqref{energybd0} follows. This completes the proof of Theorem \ref{Ori_thm0}.
\end{proof}

\bigskip
\appendix
\section{Local well-posedness} \label{local}
This section is devoted to proving the local well-posedness of \eqref{ori_sys} for large data.
We shall apply the zero-viscosity limit method.  To start with, we  consider the perturbed liquid crystal elastomers
\begin{equation}   \label{pLCE}
	\left\{\begin{aligned}
		&\d_t u-\mu\De u+u\cdot \nab u+\nab p=\nab\cdot(FF^{\top})- \div  (\nab d\odot\nab d),\\
		&\d_t F+u\cdot\nab F=\nab u F,\\
		&D_t^2 d-\De d=(-|D_t d|^2+|\nab d|^2)d,\\
		&(u,F,d,\d_t d)\big|_{t=0}=(u_0,F_0,d_0,d_1)
	\end{aligned}
	\right.
\end{equation}
with constraints
\begin{align} \label{constraints1-reA}
	&\div u=0,\quad \nab\cdot F^T=0,\quad
	\sum_{m}F_{mj}\d_m F_{ik}=\sum_m F_{mk}\d_m F_{ij},\quad
	|d|=1,
\end{align}
where $\mu>0$ is a positive constant. Then we show the following uniform estimates.

\begin{lemma}\label{E-lem}
	Let integer $k_0=3$ and $k\geq k_0$ and denote  $U=(u,F,D_t d,\nab d)$. Along the perturbed LCEs \eqref{pLCE}-\eqref{constraints1-reA}, we have the energy estimates
	\begin{align*}
		\frac{1}{2}\frac{d}{dt}\|U\|_{H^k}^2+\mu \|\nab u\|_{H^k}^2\leq C\|U\|_{H^k}^2\|U\|_{H^{k_0}}(1+\|U\|_{H^{k_0}}).
	\end{align*}
	Here the constant $C>0$ does not depend on $\mu$.
\end{lemma}
\begin{proof}
	From the system \eqref{pLCE}, we derive
	\begin{align}  \label{E-est}
		\frac{1}{2}\frac{d}{dt}\|U\|_{\dot{H}^k}^2+\mu \|\nab u\|_{\dot{H}^k}^2&=\int \d^k u\cdot \d^k \big(-u\cdot\nab u+\nab\cdot(FF^T)-\d_j(\d_j d\d_i d)\big)\\ \nonumber
		&\quad +\d^k F\cdot \d^k \big(u\cdot\nab F+\nab u F\big)-\d^k D_t d\cdot\d^k(u\cdot\nab D_t d)\\\nonumber
		&\quad +\d^k(u\cdot\nab d)\d^k \De d+\d^k D_t d\cdot \d^k [(-|D_t|^2+|\nab d|^2)d]\ dx
	\end{align}
	Due to the constriants \eqref{constraints1-reA}, we calculate a part of the right hand side as
	\begin{align}  \nonumber
		&\int \d^k u\cdot \d^k \big(\nab\cdot(FF^T)-\d_j(\d_j d\nab d)\big)+\d^k F\cdot\d^k(\nab uF)+\d^k (u\cdot\nab d)\d^k\De d \ dx\\\nonumber
		&=\int -\d^k \d_j u_i\cdot \d^k (F_{il}F_{jl})+\d^k F_{il} \cdot \d^k (\d_j u_i F_{jl})-\d^k u_i \d^k (\De d\d_i d)\\\nonumber
		&\quad -\d^k u_i \d^k \frac{1}{2}\d_i |\nab d|^2
		+\d^k (u\cdot\nab d)\d^k\De d\ dx\\ \label{NonT}
		&=\int \sum_{k_1+k_2=k;|k_1|< |k|} \d^k u_i \d^{k_1}\d_j F_{il}\d^{k_2}F_{jl}+ \d^k F_{il} (\d^{k_1}\d_j u_i \d^{k_2}F_{jl})-\d^k u_i \d^{k_1}\De d \d^{k_2}\d_i d \\\nonumber
		&\quad +\sum_{k_1+k_2=k;|k_i|< |k|} \d^{k_1}u\cdot \nab \d^{k_2} d \d^k\De d+ u\cdot \nab \d^k d\d^k\De d\ dx
	\end{align}
	The last integrand is rewritten as
	\begin{align*}
		&\int \sum_{k_1+k_2=k;|k_i|< |k|} \d^{k_1}u\cdot \nab \d^{k_2} d \d^k\De d+ u\cdot \nab \d^k d\d^k\De d\ dx\\
		&= \int -\sum_{k_1+k_2=k;|k_i|< |k|} \d_j(\d^{k_1}u\cdot \nab \d^{k_2} d )\d^k\d_j d -\d_j u\cdot\nab \d^k d \d^k \d_j d-\frac{1}{2}u\cdot\nab |\d^k\nab d|^2\ dx\\
		&\leq \sum_{k_1+k_2=k;k_i\leq k}\int |\d^k U| |\d^{k_1}U| |\d^{k_2}U|\ dx.
	\end{align*}
	The other terms in \eqref{NonT} can also be bounded by the above right hand side.
	
	Now we return to the energy estimate \eqref{E-est}, and obtain
	\begin{align*}
		&\frac{1}{2}\frac{d}{dt}\|U\|_{\dot{H}^k}^2+\mu \|\nab u\|_{\dot{H}^k}^2\\
		&\leq  \int \sum_{k_1+k_2=k;|k_1|< |k|}|\d^k U| |\d^{k_1+1}U| |\d^{k_2}U|+\sum_{k_1+k_2+k_3=k}|\d^k U||\d^{k_1}U||\d^{k_2}U||\d^{k_3}d|\ dx\\
		&\les \|U\|_{H^k}^2\|U\|_{H^{k_0}}(1+\|U\|_{H^{k_0}}).
	\end{align*}
	This completes the proof of the lemma.
\end{proof}

Then we have the following local existence result.

\begin{proof}[Proof of Theorem \ref{LWP}]
	Using the similar argument as \cite{Jiang-Luo-2018,Jiang-Liu-19}, for each $\mu>0$, one can get a unique smooth solution $(u_\mu,F_\mu,d_\mu)$ to \eqref{pLCE} on a maximal time interval $[0,T_\mu)$. Let $\tilde T_\mu\in[0,T_\mu]$ be the maximal time such that for all $t\in [0,\tilde T_\mu)$,
	\begin{align}   \label{unf-bd}
		\|(u_\mu,F_\mu,D_t d_\mu,\nab d_\mu)\|_{H^{k_0}}\leq 2(\|(u_0,F_0, d_1,\nab d_0)\|_{H^{k_0}}+1).
	\end{align}
	Then we  show that there is a uniform positive lower bound for $\tilde T_\mu$.
	
	To proceed, we have from Lemma \ref{E-lem} that
	\begin{align*}
		\frac{d}{dt}\|U_\mu(t)\|_{H^{k_0}}^2\leq 2C \|U_\mu(t)\|_{H^{k_0}}^3(1+\|U_\mu(t)\|_{H^{k_0}}).
	\end{align*}
	Denote $f_\mu(t)=1+\|U_\mu(t)\|_{H^{k_0}}$, we derive
	\begin{align*}
		\frac{d}{dt}f_\mu(t)\leq Cf_\mu^3(t),
	\end{align*}
	and hence
	\begin{align*}
		f^{-2}_\mu(0)-2Ct\leq f^{-2}_\mu(t).
	\end{align*}
	Then on the time interval $[0,\frac{1}{4Cf^2(0)}]$, we have the bound
	\begin{align*}
		f_\mu(t)\leq \sqrt{2}f_\mu(0),
	\end{align*}
	which also implies the bound \eqref{unf-bd}. Hence, we obtain a uniform positive lower bound
	\[\tilde T_\mu\geq T_0:= \frac{1}{4C(1+\|U(0)\|_{H^{k_0}})^2}.\]
	By standard arguments, we can extract a subsequence $\mu_i\rightarrow 0$ such that $(u_{\mu_i},F_{\mu_i},d_{\mu_i})$ converges smoothly to a limit $(u,F,d)$ on the interval $[0,T_0]$ which is the desired solution to the LCEs \eqref{ori_sys}. Moreover, the estimate \eqref{LE-local} is a direct consequence of \eqref{unf-bd}.
	The uniqueness can be verified directly.
	
	By Lemma \ref{E-lem} and interpolation, for $\mu=0$ and any $\si>k_0$ we obtain
	\begin{align*}
	    \frac{d}{dt}\|U\|_{H^\si}^2\leq 2\tilde C_\si C_0(1+C_0) \|U\|_{H^\si}^2 ,
	\end{align*}
	which deduces
	\[ \|U(t)\|_{H^\si}\leq e^{\tilde C_\si C_0(1+C_0)T}\|U(0)\|_{H^\si}.    \]
	This completes the proof of theorem.
\end{proof}

\bigskip
\section*{Acknowledgment}
X. Hao was supported by the NSFC Grant No. 12071043. J. Huang is supported by Beijing Institute of Technology Research Fund Program for Young Scholars, and also supported by the NSFC Grant No. 12271497. The author N. Jiang is supported by the grants from the National Natural Foundation of China under contract Nos. 11971360 and 11731008, and also supported by the Strategic Priority Research Program of Chinese Academy of Sciences, Grant No. XDA25010404.

\end{document}